\documentclass{article}

\usepackage{algorithm}
\usepackage{algorithmic}
\usepackage[utf8]{inputenc} 
\usepackage[T1]{fontenc}    
\usepackage{geometry}

\usepackage{authblk}
\newcommand{\myfootnote}[1]{
\renewcommand{\thefootnote}{}
\footnotetext{\hspace{-16.5pt}\footnotesize#1}
\renewcommand{\thefootnote}{\arabic{footnote}}}

\usepackage{multirow}
\usepackage{graphicx}
\usepackage{subfig}
\usepackage{amsmath,amssymb,amsthm}
\usepackage{xcolor}
\usepackage{url}
\usepackage[numbers,square,comma,sort&compress]{natbib}
\usepackage[hidelinks]{hyperref}

\usepackage{booktabs} 
\usepackage{soul}  
\usepackage[shortlabels]{enumitem}

\usepackage{array}
\newcolumntype{?}{!{\vrule width 1pt}}

\graphicspath{{./Images/}}

\newcommand{\bfx}{{\bf x}}
\newcommand{\bfy}{{\bf y}}
\newcommand{\bfz}{{\bf z}}
\newcommand{\bfg}{{\bf g}}
\newcommand{\bfe}{{\bf e}}
\newcommand{\bfu}{{\bf u}}
\newcommand{\calC}{{\mathcal{C}}}
\DeclareMathOperator*{\argmin}{\mathrm{arg\,min}}
\DeclareMathOperator*{\argmax}{\mathrm{arg\,max}}
\DeclareMathOperator*{\sign}{\mathrm{sign}}
\DeclareMathOperator*{\minimize}{\mathrm{minimize}}
\newcommand{\proj}{\mathrm{P}}

\theoremstyle{plain} 
\newtheorem{theorem}{Theorem}[section]
\newtheorem*{theorem*}{Theorem}
\newtheorem{lemma}[theorem]{Lemma}

\newtheorem{corollary}[theorem]{Corollary}

\theoremstyle{definition}
\newtheorem{remark}[theorem]{Remark}
\newtheorem{definition}[theorem]{Definition}

\newtheorem{assumption}{Assumption}

\newcommand{\RV}[1]{\textcolor{black}{#1}}

\begin{document}

\title{A One-bit, Comparison-Based Gradient Estimator}

\author[1]{HanQin Cai}
\author[1]{Daniel McKenzie}
\author[2]{Wotao Yin}
\author[3]{Zhenliang Zhang}
\affil[1]{Department of Mathematics, \protect\\ University of California, Los Angeles, \protect\\ Los Angeles, CA, USA.\vspace{.15cm}}
\affil[2]{Damo Academy, \protect\\ Alibaba US, \protect\\ Seattle, WA, USA.\vspace{.15cm}}
\affil[3]{Xmotors AI, \protect\\ San Jose, CA, USA.\vspace{.15cm}}
\myfootnote{\noindent Email addresses: hqcai@math.ucla.edu (H.Q. Cai), mckenzie@math.ucla.edu (D. Mckenzie), wotao.yin@alibaba-inc.com (W. Yin), and zhenliang.zhang@gmail.com (Z. Zhang).}


\maketitle

\begin{abstract}
    We study zeroth-order optimization for convex functions where we further assume that function evaluations are unavailable. Instead, one only has access to a \textit{comparison oracle}, which given two points $\bfx$ and $\bfy$ returns a single bit of information indicating which point has larger function value, $f(\bfx)$ or $f(\bfy)$. By treating the gradient as an unknown signal to be recovered, we show how one can use tools from one-bit compressed sensing to construct a robust and reliable estimator of the normalized gradient. We then propose an algorithm, coined SCOBO, that uses this estimator within a gradient descent scheme. We show that when $f(\bfx)$ has some low dimensional structure that can be exploited, SCOBO outperforms the state-of-the-art in terms of query complexity. Our theoretical claims are verified by extensive numerical experiments. 
\end{abstract}

\section{Introduction}

Consider the well-studied problem of minimizing a convex function:
\begin{equation}
 \minimize_{\bfx\in\mathbb{R}^{d}}f(\bfx) .
 \label{eq:MinProblem}
 \end{equation}
 In many applications, $\nabla f(\bfx)$ \RV{cannot be computed thus one is forced to solve \eqref{eq:MinProblem} using only (possibly noisy) function evaluations. Such problems include simulation-based optimization \cite{tran2020adadgs}, reinforcement learning \cite{mania2018simple} and hyperparameter tuning \cite{bergstra2012random}, and are variously referred to as zeroth-order, derivative free, or black-box optimization. We refer the reader to the recent survey articles \cite{Larson2019} or \cite{liu2020primer} for a modern overview of the field.}  Recently, progress in zeroth-order optimization has been made by treating $\nabla f(\bfx)$ as an unknown {\em signal} for which measurements can be acquired by finite differencing:
 \begin{equation} \label{eq:FinDiffIntro}
    y = \frac{f(\bfx+r\bfz) - f(\bfx)}{r} \approx \bfz^{\top}\nabla f(\bfx).
 \end{equation}
 One can then use LP decoding \citep{choromanski2020provably}, LASSO \citep{Wang2018} or CoSaMP \citep{cai2020zeroth} to recover a good estimator for $\nabla f(\bfx)$. In this paper, we consider the even more restrictive assumption that one only has access to $f(\bfx)$ through a \textit{comparison oracle}:
 \begin{definition}  \label{def:Comparison_Oracle}
 We say $\calC_{f}(\cdot,\cdot): \mathbb{R}^{d}\times\mathbb{R}^{d} \to \{-1,+1\}$ is a \textit{comparison oracle} for $f$ if:
 \begin{align*}
    & \mathbb{P}\left[\calC_{f}(\bfx,\bfy) = \sign\left(f(\bfy) - f(\bfx)\right)\right]  = \theta\left(|f(\bfy) - f(\bfx)|\right)
 \end{align*}
 for some non-decreasing function $\theta$ with $\theta(0) \geq 0.5$. \RV{That is, the oracle is correct with probability at least 0.5, and the probability is non-decreasing when $f(\bfy)$ and $f(\bfx)$ are more different.}
 \end{definition}
 
 \RV{We call $\sign(\bfz^{\top}\mathbf{x})\in\{-1,1\}$ a {\em  one-bit} measurement of $\mathbf{x}\in\mathbb{R}^{d}$ by $\bfz^{\top}\in\mathbb{R}^{d}$}. We push the gradient-as-signal paradigm even further by showing that comparison oracle queries can be used to construct {\em noisy} one-bit measurements of $\nabla f(\bfx)$:
 \begin{equation*}
     \calC_{f}(\bfx,\bfx + r\bfz_i) \approx \sign(f(\bfx+r\bfz_i) - f(\bfx)) \approx \sign(\bfz_{i}^{\top}\nabla f(\bfx)) \quad \text{ for } i=1,\ldots, m ,
 \end{equation*}
\RV{where the first approximation accounts for noise in the comparison oracle while the second comes from the Taylor expansion of $f$ centered at $\mathbf{x}$.} We then use tools from one-bit compressed sensing to recover an estimator, $\hat{\bfg}$, of the normalized gradient $\nabla f(\bfx)/\|\nabla f(\bfx)\|_2$. \RV{Plugging $\hat{\bfg}$ into an (inexact) Normalized Gradient Descent (NGD) scheme yields a novel algorithm, which we dub SCOBO\footnote{SCOBO stands for \underline{S}parsity-aware \underline{Co}mparison-\underline{B}ased \underline{O}ptimization. Also inspired by the Latin {\em scobo:} to seek, search, or probe.}, capable of solving Problem~\ref{eq:MinProblem} using only comparison oracle access to $f(\mathbf{x})$}. By carefully choosing $r$ (sampling radius) and $m$ (number of samples), and using the results of \citep{Plan2012}, we are able to quantify the error: $\left\|\hat{\bfg} - \nabla f(\bfx)/\|\nabla f(\bfx)\|_{2}\,\right\|_{2}$. We combine this with a novel analysis of inexact NGD to provide precise rates of convergence for SCOBO. \\
 
 \RV{Our primary motivation in studying comparison-based optimization stems from several nascent applications of optimization involving human feedback. In these applications, $f$ is internal to a human and is often subjective. As a representative example, \cite{tucker2019preference} considers maximizing the comfort of an exoskeleton used to restore mobility to individuals with lower-limb impairments by varying the exoskeleton parameters (described by $\bfx$). Here $f$ represents a user's perceived comfort, so $\nabla f(\bfx)$ is inaccessible and even attempting to evaluate $f(\bfx)$ ({\em e.g.} asking a user to assign a numerical value to their comfort) can be unreliable. However, users can usually reliably ascertain which is better: $\bfx$ or $\bfy$. That is, the human user functions as the comparison oracle $\calC_{f}(\cdot,\cdot)$. For further examples of optimization with human feedback, we refer the reader to \cite{knox2009interactively,yue2009interactively,furnkranz2012preference,wimmer2012generalization,knox2012reinforcement,christiano2017deep,tucker2020human}. In an entirely different direction, it has recently been observed that the problem of generating adversarial attacks on image classifiers from {\em hard-label} feedback can be recast as a comparison-based optimization problem \citep{cheng2019sign}. In all cases, oracle queries are expensive, either in terms of money, a human's time, or both. Hence, the most important metric when evaluating algorithms for comparison-based optimization is their {\em query complexity, i.e.} the number of oracle queries required to reach a sufficiently accurate solution.} \\
 
 \RV{The rest of this paper is laid out as follows. In the remainder of this section, we state our assumptions and notation, review prior work, and present an informal statement of our main results. In Section~\ref{sec:Prelim}, we recall results from one-bit compressed sensing and high dimensional probability that we use in the sequel. Our main theoretical contributions are in Sections 3--6. Specifically, in Section~\ref{sec:GradEst}, we present an algorithm that returns a normalized gradient estimator $\hat{\bfg} \approx \nabla f(\bfx)/\|\nabla f(\bfx)\|_2$ using only comparison oracle access to $f$, and in Section~\ref{section:INGD}, we provide novel convergence rates for inexact NGD. Section~\ref{sec:SCOBO} presents the SCOBO algorithm and a theorem bounding its query complexity while Section~\ref{section:Line_search} explores accelerating SCOBO using a {\em line search} heuristic. In Section~\ref{section:NumericalExperiments}, we present the results of various numerical experiments. Finally, Section~\ref{sec:Conclusion} contains concluding remarks while Appendix~\ref{sec:AdditionalProofs} contains several technical proofs deferred from earlier sections.}
 
\subsection{Assumptions and notation}  \label{sec:Assumptions}
We consider the following comparison oracle model.
 
\begin{definition}[Polynomial comparison oracle] \label{def:oracle 2}
We say $\calC_{f}(\cdot,\cdot)$ is a \textit{polynomial comparison oracle} for $f$ if
\label{def:PolynomialOracle}
  \begin{align*}
    & \mathbb{P}\left[\calC_{f}(\bfx,\bfy) = \sign\left(f(\bfy) - f(\bfx)\right)\right]  \RV{\geq}  \frac{1}{2} + \min\left\{\delta_0,\mu |f(\bfy) - f(\bfx)|^{\kappa - 1}  \right\},
\end{align*}
where $0<\delta_0\leq 1/2$, $\mu>0$ and $\kappa\geq 1$ are the oracle noise parameters.
\end{definition}
 
This definition of comparison oracle \RV{was initially introduced} in \citep{Jamieson2012}, and is frequently used as a model for comparisons made by humans \citep{thurstone2017law}. Informally, $\calC_{f}(\bfx,\bfy)$ tells you which point has larger function value, $f(\bfx)$ or $f(\bfy)$, with some probability of being correct. If $\kappa > 1$ 
then this probability is greater than $1/2$ and reduces to $1/2$ as $f(\bfx)\to f(\bfy)$. On the other hand, $\kappa =1$ implies that the comparison oracle is correct with constant probability, independent of $|f(\bfy) - f(\bfx)|$. For our theoretical results we restrict to the $\kappa=1$ case:
 
\begin{assumption}[Oracle model]
\label{assumption:Oracle}
$\calC_{f}(\cdot,\cdot)$ is a polynomial comparison oracle for $f$ with $\kappa=1$ and $\delta_0 < \mu \leq 0.5$. 
\end{assumption}
We now define two parametrized classes of functions. In the sequel we shall assume the function of interest belongs to one (or both) of these classes. The first class encodes latent low-dimensional structure, while the second encodes standard regularity assumptions.
 \begin{definition}[Compressible gradients]
\label{assumption:Sparsity}
For any $0<s<d$, let $CG_{s,d}$ denote the set of all functions $f: \mathbb{R}^{d} \to \mathbb{R}$ satisfying
\begin{equation*}
\|\nabla f(\bfx)\|_{1} \leq \sqrt{s}\|\nabla f(\bfx)\|_{2} \quad \text{ for all } \bfx\in\mathbb{R}^{d}
\end{equation*}
\end{definition}

This generalizes the ``sparse gradients'' assumption: $\|\nabla f(\bfx)\|_0 := |\{i: \nabla_if(\bfx) \neq 0\}| \leq s$ studied in \citep{Wang2018,Balasubramanian2018,cai2020zeroth}. We discuss this property further in Section~\ref{sec:WhyCompressibility}.
 
\begin{definition}[Regularity]
\label{assumption:Lipshitz}
\label{assumption:BoundedHessian}
$f:\mathbb{R}^d\to\mathbb{R}$ is $L$-Lipschitz differentiable if it is differentiable and 
\begin{equation*}
\|\nabla f(\bfx) - \nabla f(\bfy)\|_{2} \leq L\|\bfx - \bfy\|_{2} \quad \textnormal{ for all } \bfx,\bfy\in\mathbb{R}^{d},
\end{equation*}
is convex if
\begin{equation*}
    f(t\bfx + (1-t)\bfy) \leq tf(\bfx) + (1-t)f(\bfy) \quad \textnormal{ for all } \bfx,\bfy \in \mathbb{R}^d \text{ and } t \in [0,1],
\end{equation*}
and is $\nu$-restricted-strongly convex if 
\begin{align}\label{eq:rsc}
  f(\bfx) - f^\star \ge \RV{\frac{\nu}{2}}\|\bfx - \proj_{\star}(\bfx)\|_{2}^2 \quad \textnormal{ for all }\bfx\in \mathbb{R}^{d},
\end{align}
\RV{where $f^{\star} = \min_{\bfx\in \mathbb{R}^d}f(\bfx)$}, $\mathcal{X} = \argmin f(\bfx)$ and $\proj_{\star}(\cdot)$ is the projection operator onto this solution set: $\proj_{\star}(\bfx) = \argmin_{\bfz\in\mathcal{X}}\|\bfx - \bfz\|_{2}$. For $0< \nu \leq L$, let $\mathcal{F}_{L,\nu,d}$ denote the set of all functions $f:\mathbb{R}^{d} \to \mathbb{R}$ that are twice continuously differentiable and $L$-Lipschitz differentiable, convex, and restricted $\nu$-strongly convex.  
\end{definition}

Note that if $f\in\mathcal{F}_{L,\nu,d}$ then $\|\nabla^{2}f(\bfx)\|_{2} \leq L$ for all $\bfx\in\mathbb{R}^{d}$. We remark that if $f(\bfx)$ has sparse gradients then it cannot be strongly convex (see the appendix of \citep{cai2020zeroth}), hence restricted-strong convexity is the appropriate assumption. \RV{By \cite{zhang2015restricted}, \eqref{eq:rsc} implies:
\begin{equation}
    \|\nabla f(\bfx)\| \geq \frac{\nu}{2} \|\bfx - \proj_{\star}(\bfx)\|_2,
    \label{eq:GEB}
\end{equation}
and also
\begin{equation}
    \langle \nabla f(\bfx), \bfx - \proj_{\star}(\bfx) \rangle \geq \frac{\nu}{2} \|\bfx - \proj_{\star}(\bfx)\|_2^2.
    \label{eq:Restricted_Secant}
\end{equation}
}
Throughout this paper we use the notation $\bfg(\bfx) := \nabla f(\bfx)$, and drop the explicit dependence on $\bfx$ when it is clear from context. We say $\bfx_K$ is an {\em $\varepsilon$-optimal solution} if $f(\bfx_K) - f^{\star} \leq \varepsilon$. We also provide a table of notation for the reader's convenience, see Table~\ref{table:notation}.

\iftrue
    \begin{table}[h!]
    \caption{\RV{Table of Notation.}} \label{table:notation}
    \begin{center}
    \begin{footnotesize}
    \begin{tabular}{c|c?c|c}
    \toprule
    \textsc{Notation} & \textsc{Definition}  & 
    \textsc{Notation} & \textsc{Definition}  \\
    \midrule
    $d$                     & dimension of $\bfx$ &
    $L$                     & Lipschitz constant (Def.~\ref{assumption:Lipshitz})\\
    $f(\bfx)$               & underlying loss function &
    $\nu$                   & restricted-strong convexity (Def.~\ref{assumption:Lipshitz})\\
    $s$                     & compressibility of $\nabla f$ & $\mathcal{F}_{L,\nu,d}$ & set of functions with given properties
     \\
    $CG_{s,d}$              & set of functions with compressible gradients & $m$                     & number of samples (Alg.~\ref{alg:1BitGradEstimator})\\
    $f^\star$               & $\min\{f(\bfx):\bfx\in\mathbb{R}^d\}$, minimum value of $f$ &
    $r$                     & sampling radius (Alg.~\ref{alg:1BitGradEstimator})\\
    $\bfx_k$                & point at the $k$-th iterate &
    $\|\cdot\|_0$           & $\ell_0$-norm\\
    $\bfz_i$                & Rademacher random variable &
    $\|\cdot\|_1$           & $\ell_1$-norm\\
    $\bfg_k$                & $\nabla f(\bfx_k)$, true gradient at $\bfx_k$ &
    $\|\cdot\|_2$           & $\ell_2$-norm\\
    $\hat{\bfg}_k$          & estimated unit gradient at $\bfx_k$ &
    $\proj_{\star}(\,\cdot\,)$  & projection onto the solution set\\
    $\bfe_k$                & $\frac{\bfg_k}{\|\bfg_k\|_{2}} - \hat{\bfg}_k$, unit gradient estimation error &
    $\Delta_k$              & $\|\bfx_k - \proj_{\star}(\bfx_k)\|_{2}$, optimality gap\\
    $\calC_f(\cdot,\cdot)$      & comparison oracle for $f$ (Def.~\ref{def:Comparison_Oracle}) &
    $\mathbb{P}[\,\cdot\,]$     & probability\\
    $\delta_0$, $\mu$, $\kappa$                   &  oracle parameters (Def.~\ref{def:oracle 2}) &
    $\mathbb{E}[\,\cdot\,]$     & expectation\\
    \bottomrule
    \end{tabular}
    \end{footnotesize}
    \end{center}
    \end{table}
\else 

    \begin{table}[h!]
    \caption{Table of Notation.} \label{table:notation}
    \noindent
    \begin{minipage}[t]{.45\textwidth}
    \begin{center}
    \begin{footnotesize}
    \begin{tabular}{c|c}
    \toprule
    \textsc{Notation} & \textsc{Definition}  \\
    \midrule
    $d$                     & dimension of $\bfx$\\
    $f(\bfx)$                  & underlying loss function\\
    $s$                     & sparsity of $\nabla f$\\
    $f^\star$                & $\min\{f(\bfx):\bfx\in\mathbb{R}^d\}$\\
    $\bfx_k$                   & point at the $k$-th iterate\\
    $\bfz_i$                   & Rademacher random variable \\
    $\bfg_k$                & $\nabla f(\bfx_k)$, true gradient at $\bfx_k$ \\
    $\hat{\bfg}_k$          & estimated gradient at $\bfx_k$ \\
    $\bfe_k$    & $\frac{\bfg_k}{\|\bfg_k\|_{2}} - \hat{\bfg}_k$ \\
    $\calC_f(\cdot,\cdot)$                & comparison oracle for $f$ (Def.~\ref{def:Comparison_Oracle})\\
    $\delta_0$, $\mu$, $\kappa$                   &  oracle parameters (Def.~\ref{def:oracle 2})\\
    \bottomrule
    \end{tabular}
    \end{footnotesize}
    \end{center}
    \end{minipage}
    
    \hfill
    \noindent
    \begin{minipage}[t]{.45\textwidth}
    \begin{center}
    \begin{footnotesize}
    \begin{tabular}{c|c}
    \toprule
    \textsc{Notation} & \textsc{Definition}  \\
    \midrule
    $L$                     & Lipschitz constant (Def.~\ref{assumption:Lipshitz})\\
    $\nu$                   & restricted-strong convexity (Def.~\ref{assumption:Lipshitz})\\
    $m$                     & number of samples (Alg.~\ref{alg:1BitGradEstimator})\\
    $r$                     & sampling radius (Alg.~\ref{alg:1BitGradEstimator})\\
    $\|\cdot\|_0$           & $\ell_0$-norm\\
    $\|\cdot\|_1$           & $\ell_1$-norm\\
    $\|\cdot\|_2$           & $\ell_2$-norm\\
    $\proj_{\star}(\,\cdot\,)$   & projection onto the solution set\\ 
    $\Delta_k$ &    $\|\bfx_k - \proj_{\star}(\bfx_k)\|_{2}$ \\
    $\mathbb{P}[\,\cdot\,]$     & probability\\
    $\mathbb{E}[\,\cdot\,]$     & expectation\\
    \bottomrule
    \end{tabular}
    \end{footnotesize}
    \end{center}
    \end{minipage}
    
    \end{table}

\fi

 \subsection{Prior work}
 \label{sec:PriorWork}
 The first work to consider \eqref{eq:MinProblem} with polynomial comparison oracle feedback was \citep{Jamieson2012}. There, a coordinate descent algorithm that finds an $\varepsilon$-optimal solution (in expectation) in $\tilde{O}\big(d^{2\kappa-1}\varepsilon^{2 - 2\kappa}\big)$ queries for $\kappa >1$ and $\tilde{O}(d\log^2(\varepsilon))$ queries for $\kappa = 1$, assuming $f$ is smooth and strongly convex, where $\tilde{O}(\cdot)$ is used to suppress factors logarithmic in $d$ and the oracle parameters. We refer to this algorithm as Pairwise Comparison Coordinate Descent, or PCCD. This approach was extended by \citep{matsui2017parallel}\RV{, which} provided an empirically faster algorithm albeit with the same order of convergence. In \citep{cheng2019sign} an algorithm, SignOPT, was proposed and analyzed under an oracle model essentially equivalent to Assumption~\ref{assumption:Oracle}. To the best of our knowledge, SignOPT is the only algorithm (apart from SCOBO) which uses comparison oracle feedback to approximate the gradient. The work \citep{cheng2019sign} shows SignOPT finds an $\varepsilon$-optimal solution (again, in expectation) in $O(d^{3}\varepsilon^{-2})$ queries. We highlight the following drawbacks of existing algorithms:
 \begin{enumerate}
     \item The polynomial dependence of the number of queries on $d$ is prohibitive.
     \item Existing algorithms are not monotone. In fact, the sequence of functions values $f(\bfx_1), f(\bfx_2),\ldots$ can increase substantially before decreasing again (see Section~\ref{section:NumericalExperiments}). This makes it impossible to determine, using only comparison oracle feedback, whether one should terminate the algorithm after $k$ iterations or keep going in the hope that the sequence starts descending again.
     \item Existing results hold only in expectation.
 \end{enumerate}
 In parallel, several works \citep{carpentier2012bandit,djolonga2013high,Wang2018,Balasubramanian2018,cai2020zeroth} in {\em zeroth-order \RV{optimization}} ({\em i.e.} one has access to $f(\bfx)$ via an evaluation oracle: $E_{f}(\bfx) = f(\bfx) + \xi$, where $\xi$ is noise) have begun to consider exploiting various forms of {\em gradient sparsity} in order to reduce the dimensional dependence of the query complexity. \citep{Wang2018} and \citep{cai2020zeroth} are of particular relevance to us, as in both they use finite differencing:
 \begin{equation*}
     y_i = \frac{f(\bfx+r\bfz_i) - f(\bfx)}{r} \quad \text{ for } i=1,\ldots, m
 \end{equation*}
 to construct an underdetermined linear system $\bfy = Z\bfu$ having $\nabla f(\bfx)$ as an approximate sparse solution. Here, $Z\in\mathbb{R}^{m\times d}$ has the $\bfz_i$ as its rows. In \citep{Wang2018}, $\nabla f(\bfx)$ is approximated using LASSO \citep{tibshirani1996regression}, while work \citep{cai2020zeroth} used CoSaMP \citep{Needell2009}. We are unaware of any prior work making explicit connections between signal processing and comparison-based optimization, although we mention the work \cite{zhang2016online} which proposes to use one-bit compressed sensing for a multi-armed bandit problem under one-bit feedback. We point out their feedback model is not a comparison oracle, and they only consider linear objective functions $f$. 
 
\subsection{Why assume gradient compressibility?}
\label{sec:WhyCompressibility}
Comparison-based optimization is an unfortunate victim of the curse of dimensionality; in \citep{Jamieson2012} it is shown that the worst-case complexity of {\em any} comparison-based optimization algorithm must scale {\em at least} linearly in $d$ (for polynomial comparison oracles with $\kappa >1$ the scaling is much worse) for generic strongly convex $f(\bfx)$. In order to make progress, one needs to \RV{exploit additional structures of} $f(\bfx)$. Assumptions that encode low intrinsic dimension, such as gradient sparsity/compressibility,  multi-ridge structure ({\em i.e.} $f(\mathbf{x}) = g(A\mathbf{x})$ for some $A\in\mathbb{R}^{k\times d}$ and $g: \mathbb{R}^{k}\to\mathbb{R}$ with $k \ll d$) or the existence of active subspaces \cite{constantine2015active} have successfully been incorporated into other derivative-free contexts \citep{carpentier2012bandit,djolonga2013high,wang2016bayesian,Wang2018,Balasubramanian2018,golovin2019gradientless,choromanski2019complexity,cai2020zeroth}. Moreover, this phenomenon is often observed in applications such as hyperparameter tuning for neural networks \citep{bergstra2012random} and combinatorial optimization \citep{hutter2014efficient}, as well as simulation based optimization \citep{knight2007association,constantine2015active, cartis2020dimensionality}. Combining gradient compressibility with comparison-based optimization is, therefore, a natural step forward.
 
 \subsection{Our contributions}
 In this paper, we provide an algorithm, SCOBO, for comparison oracle optimization which overcomes the three shortcomings mentioned in Section~\ref{sec:PriorWork}. The key innovation in SCOBO is a novel gradient estimator which uses tools from one-bit compressed sensing.    
 \begin{theorem*}[Main results, informally stated] \label{thm:InformalMainTheorem}
Suppose $f\in \mathcal{F}_{L,\nu,d}$, $f\in CG_{s,d}$, and the comparison oracle $\calC_{f}(\cdot,\cdot)$ satisfies Assumption~\ref{assumption:Oracle}. Then SCOBO (Algorithm~\ref{algorithm:SCOBO}) finds an $\varepsilon$-optimal solution {\em with high probability} in $\tilde{O}\left(s\varepsilon^{-3/2}\delta_{0}^{-2}\right)$ queries. Moreover, $f(\bfx_{k}) - f(\bfx_{k-1}) \leq 0$ \RV{with high probability for all $k$.}
\end{theorem*}
 By assuming $f(\bfx)$ has some low dimensional structure, we are able to reduce the query complexity to only {\em logarithmic dependence on $d$}. While the polynomial dependence on $1/\varepsilon$ is undesirable, in practice, this can be avoided by using an appropriate {\em line search} heuristic, which we introduce in Section~\ref{section:Line_search}. In Section~\ref{sec:Synthetic} we benchmark SCOBO against the state-of-the-art algorithms discussed in Section~\ref{sec:PriorWork}, and find that it offers a substantial speed-up for $f \in CG_{s,d}$. Finally, we end with some promising results of SCOBO applied to real-world problems from the MuJoCo suite \citep{todorov2012mujoco}.

\section{Preliminaries}
\label{sec:Prelim}
In this section, we collect some well-known results from the literature for later use. 
\subsection{One-bit compressed sensing}
\label{section:One_Bit_Compressed_Sensing}
One-bit compressed sensing, first introduced in \citep{Boufounos2008}, is a framework for recovering an unknown signal from highly quantized linear measurements. Specifically, we assume that $\bfx\in\mathbb{R}^{d}$ is unknown and that we only have access to measurements $y_{1},\ldots, y_{m} \in \{-1,+1\}$ which are correlated with $\sign(\bfz_{i}^{\top}\bfx)$. In the {\em noise-free} setting we assume that $y_{i} = \sign(\bfz_{i}^{\top}\bfx)$. More generally, we assume that $y_i = \xi_i\sign(\bfz_i^{\top}\bfx)$ where $\xi_{i}\in\{-1,1\}$ and $\mathbb{P}[\xi_i = 1] = p >1/2$ allows for a random bit flip. Remarkably, even in the presence of corruptions, one can still recover $\bfx$ from the measurement vector $\bfy = [y_1,\ldots, y_{m}]^{\top}\in\{0,1\}^{m}$, as the following theorem quantifies. For notational convenience, we set $\tilde{y}_i := \sign(\bfz_i^{\top}\bfx)$. By $\mathcal{U}(\mathbb{S}^{d-1})$ we mean the uniform distribution on the unit sphere $\mathbb{S}^{d-1}\subset\mathbb{R}^d$. If a random vector $\mathbf{z}$ is sampled from $\mathcal{U}(\mathbb{S}^{d-1})$ we write $\mathbf{z} \sim \mathcal{U}(\mathbb{S}^{d-1})$.

\begin{theorem}[{\citep[Corollary~3.1]{Plan2012}}]
\label{theorem:PlanVershynin}
Suppose $\mathbf{z}_i \sim \mathcal{U}(\mathbb{S}^{d-1})$ independently for $i=1,\ldots, m$. Suppose that $\|\bfx\|_1\leq \sqrt{s}$ and $\|\bfx\|_{2} = 1$. If $y_i = \xi_i\tilde{y}_i$ with $\xi_i\in\{-1,1\}$ i.i.d. and $\mathbb{P}[\xi_i = 1] = p$, then
\begin{equation}
    \hat{\bfx} := \argmax_{\|\bfx'\|_{1} \leq \sqrt{s} \textnormal{ and }\|\bfx'\|_{2} \leq 1} \sum\nolimits_{i=1}^{m} y_{i}\bfz_{i}^{\top}\bfx' 
\end{equation}
satisfies $\|\hat{\bfx} - \bfx\|_{2} \leq \sqrt{\delta}$ with probability at least $1 - 8\exp\left(-c\delta^{2}m\right)$ as long as:
$$
m \geq C\delta^{-2}(p-1/2)^{-2}s\log(2d/s).$$ 
\end{theorem}

\begin{remark}
The theorem presented in \citep[Corollary~3.1]{Plan2012} is for $\bfz_i$ Gaussian random vectors. However, one can check that the result holds for any rotationally invariant distribution.
\end{remark}

\subsection{High-dimensional probability}
\label{sec:HighDimProb}
Analysis of a random vector $\mathbf{z}\sim \mathcal{U}(\mathbb{S}^{d-1})$ for large $d$ is a key ingredient to our theoretical guarantees. For the sake of completeness, we also include proofs in this section.

\begin{theorem}
\label{thm:HighDimSphereTheorem}
Let $\bfz\sim \mathcal{U}(\mathbb{S}^{d-1})$, and let $z_i$ be the $i$-th component of $\bfz$. Then:
\begin{enumerate}
    \item $\displaystyle \mathbb{E}\left[z_i\right] = 0$.
    \item $\displaystyle \mathbb{E}\left[z_i^{2}\right] = 1/d$.
    \item $\displaystyle \mathbb{P}\left[\left|z_i\right| \geq 1/\sqrt{d}\right] \geq 1/2$.
\end{enumerate}
\end{theorem}

\begin{proof}
{\bf Part 1.} Without loss of generality, we may assume that $i=1$.
 Since the distribution of $z_1$ is symmetric about the origin, it follows that $\mathbb{E}[z_i] = 0$. 
    
{\bf Part 2.} Again, we may assume that $i=1$ without loss of generality. The probability of $z_1 > h$ is proportional to the area of the {\em hyperspherical cap} of height $h$. That is, the area of the portion of $\mathbb{S}^{d-1}$ above the hyperplane with equation $x_1 = h$. From \citep{li2011concise}, we get that:
\begin{align*}
\mathbb{P}[z_1 \geq h] & = \frac{\textnormal{Area hyperspherical cap of height } h}{\textnormal{Area of } \mathbb{S}^{d-1}} \\
    & = \frac{1}{2} I_{1-h^2}\left(\frac{d-1}{2},\frac{1}{2}\right),
\end{align*}
where $I$ represents the regularized, incomplete Beta function. Equivalently, $X = 1 - z_1^{2}$ is a $\mathrm{Beta}\left(\frac{d-1}{2},\frac{1}{2}\right)$ random variable, hence:
\begin{align*}
\mathbb{E}[z_1^{2}] & = 1 - \mathbb{E}[X] = 1 -  \left(\frac{(d-1)/2}{(d-1)/2 + 1/2}\right) 
 = 1 - \frac{d-1}{d} = \frac{1}{d}.
\end{align*} 

{\bf Part 3.} From the above:
\begin{equation*}
\mathbb{P}[z_1 \geq 1/\sqrt{d}] = \frac{1}{2} I_{1-1/d}\left(\frac{d-1}{2},\frac{1}{2}\right).
\end{equation*}
We note, as in \citep{golovin2019gradientless}, that the function $d\to I_{1-1/d}\left(\frac{d-1}{2},\frac{1}{2}\right)$ is increasing. Because:
\begin{equation*}
I_{1-1/2}\left(\frac{2-1}{2},\frac{1}{2}\right) = I_{1/2}\left(\frac{1}{2},\frac{1}{2}\right) = \frac{1}{2},
\end{equation*}
where for the second equality we have used the fact that the distribution $\mathrm{Beta}\left(\frac{1}{2},\frac{1}{2}\right)$ is equal to the arcsine distribution. The claim then follows by symmetry, as:
\begin{align*}
\mathbb{P}\left[|z_1|\geq \frac{1}{\sqrt{d}}\right] = 2\mathbb{P}\left[z_1\geq \frac{1}{\sqrt{d}}\right] &= I_{1-1/d}\left(\frac{d-1}{2},\frac{1}{2}\right) \geq \frac{1}{2}. 
\end{align*}
\end{proof}

\section{A one-bit gradient estimator}
\label{sec:GradEst}

The construction of our gradient estimator was inspired by the observation:
\begin{equation}
\underbrace{\calC_{f}(\bfx,\bfx + r\bfz_i)}_{\RV{=:y_i}} \stackrel{(a)}{\approx} \underbrace{\sign(f(\bfx+r\bfz_i) - f(\bfx))}_{\RV{=:\hat{y}_i}} \stackrel{(b)}{\approx} \underbrace{\sign(\bfz_{i}^{\top}\bfg)}_{\RV{=:\tilde{y}_i}},
\label{eq:BigIdea}
\end{equation}
where $r >0$ and $\bfz_i\in\mathbb{R}^{d}$ is a random perturbation. Thus, one may think of the $y_i = \calC_{f}(\bfx,\bfx + r\bfz_i)$ as {\em approximate one-bit measurements of $\bfg$}. Hence, one may use one-bit compressed sensing, as outlined in Section~\ref{section:One_Bit_Compressed_Sensing}, to recover $\bfg$ from $\bfy = [y_1,\ldots, y_m]^{\top}\in\mathbb{R}^{m}$. We present the resulting gradient estimation algorithm as Algorithm~\ref{alg:1BitGradEstimator}. Analysing the accuracy of Algorithm~\ref{alg:1BitGradEstimator} requires quantifying the approximations (a) and (b) in \eqref{eq:BigIdea}, which we do in Section \ref{sec:MeasurementError}. With this in hand, we are able to use the results of Section \ref{section:One_Bit_Compressed_Sensing} to quantify the approximation error: $\left\| \hat{\bfg} - \left(\bfg/\|\bfg\|_2\right)\right\|_{2}$. We present this result in Section \ref{sec:ReconstructionError}.

\begin{algorithm}
\caption{1BitGradEst} \label{alg:1BitGradEstimator}
\begin{algorithmic}[1]
    \STATE {\bf Inputs:} $\bfx$: Current point, $s$: target sparsity, $m$: number of queries, $r$: sampling radius
    \STATE Generate $\bfz_1,\ldots,\bfz_m \sim \mathcal{U}(\mathbb{S}^{d-1})$. \\
    \STATE $y_i \gets \calC_{f}(\bfx,\bfx+r\bfz_i)$ for $i=1,\ldots, m$
    \STATE Solve the convex program:
    \begin{equation}
    \hat{\bfg} \gets \argmax_{\|\bfg'\|_1 \leq \sqrt{s} \textnormal{ and } \|\bfg'\|_{2} \leq 1}\sum\nolimits_{i=1}^{m}y_i\bfz_{i}^{\top}\bfg' 
    \label{eq:QuadProg}
    \end{equation}
    \STATE {\bf Output:} $\hat{\bfg}$
\end{algorithmic}
\end{algorithm}

\subsection{Quantifying the error in measurement}\label{sec:MeasurementError}
Recall $y_i$, $\hat{y}_i$, and $\tilde{y}_i$ from \eqref{eq:BigIdea}. The goal of this section is to quantify the probability of $y_i=\tilde{y}_i$.
We have the following lemma relating them. 

\begin{lemma}
$\mathbb{P}[y_i = \tilde{y}_i] = \mathbb{P}[y_i = \hat{y}_i \text{ and } \hat{y}_i = \tilde{y}_i] + \mathbb{P}[y_i = -\hat{y}_i \text{ and } \hat{y}_i = -\tilde{y}_i]$ and similarly, for any event $\mathcal{E}$,  $\mathbb{P}[y_i = \tilde{y}_i |\mathcal{E}] = \mathbb{P}[y_i = \hat{y}_i \text{ and } \hat{y}_i = \tilde{y}_i|\mathcal{E}] + \mathbb{P}[y_i = -\hat{y}_i \text{ and } \hat{y}_i = -\tilde{y}_i|\mathcal{E}]$. 
\label{lemma:Mutually_Exclusive_Events}
\end{lemma}

\begin{proof}
Clearly, $y_i = \tilde{y}_i$ if $y_i = \hat{y}_i$ and $\hat{y}_i = \tilde{y}_i$, but we also have $y_i = \tilde{y}_i$ when $y_i = -\hat{y}_i$ and $\hat{y}_i = -\tilde{y}_i$ (these are the only possibilities as $y_i,\hat{y}_i$ and $\tilde{y}_i$ are binary random variables). These events are mutually exclusive, so $\mathbb{P}[y_i = \tilde{y}_i] = \mathbb{P}[y_i = \hat{y}_i \text{ and } \hat{y}_i = \tilde{y}_i] + \mathbb{P}[y_i = -\hat{y}_i \text{ and } \hat{y}_i = -\tilde{y}_i]$ as claimed. The proof for the case conditioned on $\mathcal{E}$ is similar.
\end{proof}
Our strategy is to define an event $\mathcal{B}$ such that $\mathbb{P}[y_i= \tilde{y}_i|\mathcal{B}] \geq 0.5 + \gamma$ and $\mathbb{P}[y_i = \tilde{y}_i|\mathcal{B}^{c}] \geq 0.5$, where $\gamma > 0$ is a small constant. We then use
\begin{equation}
    \begin{split}
    \mathbb{P}\left[y_i = \tilde{y}_i\right] & = \mathbb{P}\left[y_i = \tilde{y}_i\middle|\mathcal{B} \right]\mathbb{P}[\mathcal{B}] 
+ \mathbb{P}\left[y_i =\tilde{y}_i \middle|~ \mathcal{B}^{c} \right]\left(1 - \mathbb{P}[\mathcal{B}]\right) \\
    &\geq \left(0.5 +\gamma\right)\mathbb{P}[\mathcal{B}] + 0.5\left(1 - \mathbb{P}[\mathcal{B}]\right) \\
    & = 0.5 + \gamma \mathbb{P}[\mathcal{B}]. 
\end{split}
\label{eq:outlineStrategy}
\end{equation}
We begin with Taylor's theorem:
\begin{equation}
f(\bfx + r\bfz_{i}) - f(\bfx) = r \bfz_{i}^{\top}\bfg + \frac{1}{2}r^{2}\bfz_{i}^{\top}\nabla^{2}f(\bfx+t_0\bfz_{i})\bfz_{i} \quad \text{ for some } t_0 \in (0,1)
\label{eq:Taylor_Expansion}
\end{equation}
Let $e_i = e_i(\bfx,t_0;\bfz_i):= \bfz_i^{\top}\nabla^{2}f(\bfx+t_0\bfz_i)\bfz_i$. If $f$ is convex then $\nabla^{2}f(\bfx+t_0\bfz_i)$ is positive semi-definite and hence $e_i \geq 0$. 

\begin{lemma} \label{lemma:FirstCoordSphericalDist}
Suppose that $\bfz_i\sim \mathcal{U}(\mathbb{S}^{d-1})$, then
$\mathbb{P}\left[\left|\bfz_i^{\top}\bfg\right| \geq \|\bfg\|_2/\sqrt{d}\right] \geq \frac{1}{2} $.
\end{lemma}

\begin{proof}
Without loss of generality, assume $\bfg = \mathbf{c}_1\|\bfg\|_2 $, where $\mathbf{c}_1$ is the first canonical basis vector. Then:
\begin{align*}
\mathbb{P}\left[\left|\bfz_i^{\top}\bfg\right| \geq \|\bfg\|_2/\sqrt{d}\right] = \mathbb{P}\left[\left|\bfz^{\top}_{i}\mathbf{c}_1\right| \geq 1/\sqrt{d}\right] 
= \mathbb{P}[\left|z_{i,1}\right| \geq 1/\sqrt{d}] \geq \frac{1}{2},
\end{align*}
where the final inequality is from Theorem~\ref{thm:HighDimSphereTheorem} Part 3.
\end{proof}



\begin{lemma}
For any $a,b\in\mathbb{R}\setminus\{0\}$, if $|a| - |b| >0$ then $\sign(a) = \sign(a+b)$
\label{lemma:ab_lemma}
\end{lemma}

\begin{proof}
If $a,b>0$ or $a,b < 0$ there is nothing to prove. So, suppose $a>0$ while $b<0$. Here, $\sign(a+b) = \sign(|a| - |b|) = +1 = \sign(a)$ using the assumption $|a| - |b| >0$. The case $a<0$, $b >0$ is similar.
\end{proof}

\begin{lemma}
Suppose $f\in \mathcal{F}_{L,\nu,d}$ and $\bfz_i\sim \mathcal{U}(\mathbb{S}^{d-1})$. \RV{Fix any $\varepsilon >0$}, set $r = \varepsilon\nu/(2L\sqrt{d})$ and suppose $\|\bfg\|_2 \geq \varepsilon\nu/2$. Define the event $\mathcal{B} := \{\bfz_i: \ |\bfz_i^{\top}\mathbf{g}| \geq \frac{\varepsilon\nu}{2\sqrt{d}}\}$. Then $\mathbb{P}[\mathcal{B}] \geq 0.5$ and $\mathbb{P}[\hat{y}_i = \tilde{y}_i|\mathcal{B}] = 1$.
\label{lemma:hat_y=tilde_y}
\end{lemma}

\begin{proof}

From Lemma~\ref{lemma:FirstCoordSphericalDist} if $\|\bfg\|_2 \geq \varepsilon\nu/2$ then
\begin{equation}
    \mathbb{P}[\mathcal{B}] := \mathbb{P}\left[|\bfz_i^{\top}\mathbf{g}| \geq \frac{\varepsilon\nu}{2\sqrt{d}}\right] \geq \mathbb{P}\left[|\bfz_i^{\top}\mathbf{g}| \geq \frac{\|\bfg\|_2}{\sqrt{d}}\right] \geq \frac{1}{2}.
\end{equation}
$f$ is $L$-Lipschitz differentiable (as $f\in \mathcal{F}_{L,\nu,d}$), hence $\RV{|e_i|}\leq \|\nabla^{2}f(\bfx)\|_2 \leq L$. So, if $\mathcal{B}$ occurs and $r$, $\|\bfg\|_2$ are as stated then
\begin{equation}
   r|\bfz_i^{\top}\mathbf{g}| - \frac{r^2}{2}|e_i| \geq \left(\frac{\varepsilon\nu}{2L\sqrt{d}}\right)\left(\frac{\varepsilon\nu}{2\sqrt{d}}\right) - \frac{\varepsilon^2\nu^2}{4L^2d}\left(\frac{L}{2}\right) = \frac{\varepsilon^2\nu^2}{8Ld} > 0.
\end{equation}
From \eqref{eq:Taylor_Expansion} 
\begin{equation}
    \hat{y}_i := \sign(f(\bfx+r\bfz_i) - f(\bfx)) = \sign\left(r\bfz_i^{\top}\mathbf{g} +\frac{r^2}{2}e_i \right) .
    \label{eq:hat_tilde_1}
\end{equation}
Applying Lemma~\ref{lemma:ab_lemma} with $a= r\bfz_i^{\top}\mathbf{g}$ and $b=r^2e_i/2$ as $r|\bfz_i^{\top}\mathbf{g}| - r^2|e_i|/2 >0$ we have:
\begin{equation}
    \sign\left(r\bfz_i^{\top}\mathbf{g} +\frac{r^2}{2}e_i \right) = \sign\left(r\bfz_i^{\top}\mathbf{g}\right) =: \tilde{y}_i
    \label{eq:hat_tilde_2}
\end{equation}
and so (combining \eqref{eq:hat_tilde_1} and \eqref{eq:hat_tilde_2}) $\hat{y}_i = \tilde{y}_i$.
\end{proof}

\begin{lemma}
Suppose $f\in \mathcal{F}_{L,\nu,d}$ and the comparison oracle $\calC_{f}(\cdot,\cdot)$ satisfies Assumption~\ref{assumption:Oracle}. If $\mathcal{B}$ is as in Lemma~\ref{lemma:hat_y=tilde_y}, then $\mathbb{P}[y_i = \tilde{y}_i|\mathcal{B}] \geq 0.5 + \delta_0$.
\label{thm:Conditional_on_A}
\end{lemma}

\begin{proof}
When $\kappa=1$ the accuracy of a polynomial comparison oracle is independent of $\left|f(\bfx + r\bfz_{i}) - f(\bfx)\right|$ and so is independent of $\bfz_i$. Thus $\mathbb{P}[y_i = \hat{y}_i|\mathcal{B}] = \mathbb{P}[y_i = \hat{y}_i] = 0.5 + \delta_0$. Appealing to Lemma~\ref{lemma:Mutually_Exclusive_Events}:
\begin{equation*}
    \mathbb{P}[y_i = \tilde{y}_i|\mathcal{B}] \geq \mathbb{P}\left[y_i = \hat{y}_i \text{ and } \hat{y}_i = \tilde{y}_i|\mathcal{B}\right]. 
\end{equation*}
The event $\hat{y}_i = \tilde{y}_i$ depends on $\bfz_i$ only whereas the event $y_i = \hat{y}_i$ depends on the randomness inherent in the oracle only. Hence, these two events are independent, even when conditioned on $\mathcal{B}$. So, 
\begin{equation}
    \mathbb{P}\left[y_i = \hat{y}_i \text{ and } \hat{y}_i = \tilde{y}_i|\mathcal{B}\right] = \mathbb{P}\left[y_i = \hat{y}_i |\mathcal{B}\right]\mathbb{P}\left[\hat{y}_i = \tilde{y}_i|\mathcal{B}\right] \geq 0.5 + \delta_0,
\end{equation}
where $\mathbb{P}\left[\hat{y}_i = \tilde{y}_i|\mathcal{B}\right] = 1$ is from Lemma~\ref{lemma:hat_y=tilde_y}
\end{proof}
We now focus our attention on the event $\mathcal{B}^{c}$. 

\begin{lemma}
Suppose $f\in \mathcal{F}_{L,\nu,d}$ and the comparison oracle $\calC_{f}(\cdot,\cdot)$ satisfies Assumption~\ref{assumption:Oracle}. If $\mathcal{B}$ is as in Lemma~\ref{lemma:hat_y=tilde_y}, then $\mathbb{P}\left[y_i = \tilde{y}_i|\mathcal{B}^{c}\right] \geq 0.5$.
\label{lemma:LowerBound_on_A_c}
\end{lemma}

\begin{proof}
As above, the events $\hat{y}_i = \tilde{y}_i$ and $y_i = \hat{y}_i$ are independent conditioned on $\mathcal{B}^c$. From Lemma~\ref{lemma:Mutually_Exclusive_Events}:
\begin{equation}
    \mathbb{P}[y_i = \tilde{y}_i |\mathcal{B}^{c}] = \mathbb{P}[y_i = \hat{y}_i \text{ and } \hat{y}_i = \tilde{y}_i |\mathcal{B}^{c}] + \mathbb{P}[y_i = -\hat{y}_i \text{ and } \hat{y}_i = -\tilde{y}_i|\mathcal{B}^{c}].
    \label{eq:WriteProbAsSum}
\end{equation}
Let $p_{1}:= \mathbb{P}\left[\hat{y}_i = \tilde{y}_i|\mathcal{B}^{c}\right]$ and $p_{2}:= \mathbb{P}\left[y_i = \hat{y}_i|\mathcal{B}^{c}\right]$. As in the proof of Lemma~\ref{thm:Conditional_on_A}, conditioning on $\mathcal{B}^{c}$ has no effect on the event $y_i = \hat{y}_i$ so 
\begin{equation*}
    p_{2} = \mathbb{P}\left[y_i = \hat{y}_i|\mathcal{B}^{c}\right] = \mathbb{P}\left[y_i = \hat{y}_i\right] \geq 0.5 + \delta_0 \geq 0.5.
\end{equation*}
Write $\mathcal{B}^{c}$ as the union of two disjoint sets, {\em i.e.} $\mathcal{B}^{c} = \mathcal{B}_1\cup\mathcal{B}_2$ where
\begin{equation}
    \mathcal{B}_1 = \left\{\bfz_i: -\frac{\varepsilon\nu}{2\sqrt{d}} < \bfz_i^{\top}\mathbf{g} < 0\right\} \quad \text{ and } \quad \mathcal{B}_2 = \left\{\bfz_i: 0 < \bfz_i^{\top}\mathbf{g} < \frac{\varepsilon\nu}{2\sqrt{d}}\right\}.
\end{equation}
By symmetry of the distribution $\mathcal{U}(\mathbb{S}^{d-1})$, $\mathbb{P}[\mathcal{B}_1] = \mathbb{P}[\mathcal{B}_2] = 0.5\mathbb{P}[\mathcal{B}^{c}]$. As $f$ is convex (because $f\in \mathcal{F}_{L,\nu,d}$) $e_i \geq 0$ always. Thus, when $\bfz_i^{\top}\mathbf{g} >0$ ({\em i.e.} when $\mathcal{B}_2$ occurs) we have:
\begin{equation}
    \hat{y}_i = \sign\left(r\bfz_i^{\top}\mathbf{g} +\frac{r^2}{2}e_i \right) = \sign\left( \bfz_i^{\top}\mathbf{g}\right) = \tilde{y}_i
\end{equation}
and so:
\begin{align*}
    p_1 &= \mathbb{P}\left[\hat{y}_i = \tilde{y}_i \middle | \mathcal{B}^{c} \right] = \mathbb{P}\left[\hat{y}_i = \tilde{y}_i\middle | \mathcal{B}_1 \right] \frac{\mathbb{P}[\mathcal{B}_1]}{\mathbb{P}[\mathcal{B}]} +  \mathbb{P}\left[\hat{y}_i = \tilde{y}_i\middle | \mathcal{B}_2 \right] \frac{\mathbb{P}[\mathcal{B}_2]}{\mathbb{P}[\mathcal{B}]} \\
    & \geq \mathbb{P}\left[\hat{y}_i = \tilde{y}_i\middle | \mathcal{B}_2 \right] \frac{\mathbb{P}[\mathcal{B}_2]}{\mathbb{P}[\mathcal{B}]} \\
    & \geq (1)(0.5) = 0.5 .
\end{align*}
Using independence and appealing to equation \eqref{eq:WriteProbAsSum}:
\begin{align*}
\mathbb{P}[y_i = \tilde{y}_i |\mathcal{B}^{c}] = p_1p_2 + (1-p_1)(1-p_2) = 1 + 2p_1p_2 - p_1 - p_2 := f(p_1,p_2) .
\end{align*}
Observe $\partial f/\partial p_{1} = 2p_{2}-1$ and $\partial f/\partial p_{2} = 2p_{1} -1$. Hence $f(p_1,p_2)$ is increasing in $p_{1},p_{2}$ for $p_{1},p_{2} \geq 0.5$ which holds by the above. Moreover $f(0.5,0.5) = 0.5$, whence $\mathbb{P}\left[y_i = \tilde{y}_i|\mathcal{B}^{c}\right] \geq 0.5$.
\end{proof}

\begin{remark}
The independence of the events $\hat{y}_i = \tilde{y}_i$ and $y_i = \hat{y}_i$ is essential in proving Lemma~\ref{lemma:LowerBound_on_A_c}. This independence no longer holds for polynomial comparison oracles with $\kappa >1$. This is the key obstruction to extending our theory to this case.
\end{remark}

\begin{lemma}
\label{lemma:Bound_Prob_of_Failure}
Suppose $f\in \mathcal{F}_{L,\nu,d}$ and the comparison oracle $\calC_{f}(\cdot,\cdot)$ satisfies Assumption~\ref{assumption:Oracle}. \RV{Fix any $\varepsilon >0$}, set $r = \varepsilon\nu/(2L\sqrt{d})$ and suppose $\|\bfg\|_2 \geq \varepsilon\nu/2$. Then
      $\displaystyle \mathbb{P}[y_i = \tilde{y}_i] \geq 0.5 + 0.5\delta_0$. 
\end{lemma}

\begin{proof}
We use the strategy outlined at the beginning of this sub-section with $\mathcal{B} := \{\bfz_i: \ |\bfz_i^{\top}\mathbf{g}| \geq \frac{\varepsilon\nu}{2\sqrt{d}}\}$. From Lemma~\ref{thm:Conditional_on_A} we have $\mathbb{P}[y_i = \tilde{y}_i|\mathcal{B}] \geq 0.5 + \delta_0$. Because $\mathbb{P}[y_i = \tilde{y}_i|\mathcal{B}^{c}] \geq 0.5$ (by Lemma~\ref{lemma:LowerBound_on_A_c}) from \eqref{eq:outlineStrategy} we obtain:
\begin{equation}
    \mathbb{P}[y_i = \tilde{y}_i] \geq 0.5 + \delta_0\mathbb{P}[\mathcal{B}].
\end{equation}
Finally, Lemma~\ref{lemma:hat_y=tilde_y} yields $\mathbb{P}[\mathcal{B}] \geq 0.5$, proving this lemma.
\end{proof}

\subsection{Quantifying the accuracy of reconstruction}
\label{sec:ReconstructionError}

With Lemma~\ref{lemma:Bound_Prob_of_Failure} in hand we may now quantify how close $\hat{\bfg}$ is to the normalized true gradient:

\begin{theorem}
\label{theorem:Grad_Estimate_Error_Bound}
Suppose $f\in \mathcal{F}_{L,\nu,d}$, $f\in CG_{s,d}$, and the comparison oracle $\calC_{f}(\cdot,\cdot)$ satisfies Assumption~\ref{assumption:Oracle}. Fix a target solution accuracy $\varepsilon >0$ and a target gradient accuracy $0<\eta < 1$. Suppose  $\|\bfg\|_{2} \geq \varepsilon\nu/2$ and let $\hat{\bfg}$ denote the output of Algorithm~\ref{alg:1BitGradEstimator} with parameters
    \begin{align*}
        m \geq 4C\RV{\eta^{-4}}\delta_0^{-2}s\log(2d/s) = O\left(\delta_0^{-2}s\log(d)\right) \quad\text{and}~
        r = \frac{\varepsilon\nu}{2L\sqrt{d}}.
    \end{align*}
Then $\left\|\hat{\bfg} - \frac{\bfg}{\|\bfg\|_{2}}\right\|_{2} \leq \eta$ holds with probability at least $1 - 8\exp\left(-c\eta^{4}s\log(2d/s)\right)$.
\end{theorem}

\begin{proof} By Lemma~\ref{lemma:Bound_Prob_of_Failure} we may write $y_i = \xi_i\tilde{y}_i$ where the $\xi_i$ are Bernoulli random variables with $\mathbb{P}[\xi_i = 1] \geq 0.5 + \delta_0/2$ and $\mathbb{P}\left[\xi_i = -1\right] \leq 0.5 - \delta_0/2$, as long as $\|\bfg\|\geq \varepsilon\nu$. Conditioned on $\bfx$, the $\xi_i$ are independent. So, substituting $p = 0.5 - \delta_0/2$ and \RV{$\eta = \sqrt{\delta}$} in Theorem~\ref{theorem:PlanVershynin} we obtain $\displaystyle \left\|\hat{\bfg} - \frac{\bfg}{\|\bfg\|_{2}}\right\|_{2} \leq \eta$ with the stated probability as long as
\begin{align*}
m & \geq 4C\eta^{-4}\delta_0^{-2}s\log(2d/s),
\end{align*}
thus proving the theorem.
\end{proof}

\section{Inexact normalized gradient descent}
\label{section:INGD}
\begin{algorithm}
\caption{INGD}
\begin{algorithmic}[1]
    \STATE {\bf Inputs:} $\bfx_0:$ Initial point, $\alpha:$ step size, $K$: number of iterations, $\eta$: target gradient accuracy
    \FOR{ \RV{$k=0,\ldots, K-1$} }
        \STATE Obtain $\hat{\bfg}_k$ with $\left\|\hat{\bfg}_k - \frac{\bfg_k}{\|\bfg_k\|_{2}} \right\|_{2} \leq \eta$
        \STATE $\bfx_{k+1} = \bfx_{k} - \alpha \hat{\bfg}_k$
    \ENDFOR
    \STATE {\bf Output:} $\bfx_{K}$
\end{algorithmic}
\label{algorithm:INGD}
\end{algorithm}
\RV{Theorem~\ref{theorem:Grad_Estimate_Error_Bound} shows Algorithm~\ref{alg:1BitGradEstimator} reliably finds an estimate of the {\em normalized gradient}: $\hat{\bfg} \approx \bfg/\|\bfg\|_2$. The {\em gradient magnitude} $\|\bfg\|_2$, however, cannot be recovered via a one-bit approach. Thus, we cannot naively use $\hat{\bfg}$ within a gradient descent framework. Instead, we are led to consider \textit{normalized gradient descent} (NGD).} (Exact) NGD, defined by the iteration $\bfx_{k+1} = \bfx_{k} - \bfg_k/\|\bfg_k\|_{2}$, was first analyzed in \citep{nesterov1984minimization}, where it was suggested as an algorithm for quasi-convex minimization. Recently, there has been renewed interest in NGD from the machine learning community, as it has been shown that NGD can efficiently avoid saddle points \citep{levy2016power} as well as deal with issues of exploding gradients \citep{yu2017normalized}. However, most work in this area assumes one has noise-free access to $\bfg_k$, although see \citep{hazan2015beyond} for an interesting stochastic extension of NGD to the empirical risk minimization problem. To the best of our knowledge, there is no prior work on {\em inexact} NGD (INGD), where one only has access to a biased estimator $\hat{\bfg}_k \approx \bfg_k/\|\bfg_k\|_{2}$ satisfying $\|\hat{\bfg}_k - \bfg_k/\|\bfg_k\|_{2}\,\|_{2} \leq \eta$ \RV{with high probability}. We consider this situation, and prove the following theorem:

\begin{theorem}
\label{thm:INGD_Convergence}
Suppose $f\in \mathcal{F}_{L,\nu,d}$. Choose any step size $\alpha > 0$ and target gradient accuracy \RV{$0< \eta < \nu/(2L)$}. \RV{Recall $\proj_{\star}(\cdot)$ is the projection operator on to the solution set $\mathcal{X} = \argmin_{\bfx\in\mathbb{R}^{d}} f(\bfx)$ and} define $\Delta_k = \|\bfx_k  - \proj_{\star}(\bfx_k)\|_2$ and:
\begin{equation}
\RV{\rho^{\star} = \frac{1-\eta^{2}}{\frac{\nu}{2L} - \eta}}. \label{eq:def_rho_*}
\end{equation}
Suppose $\|\hat{\bfg}_k - \bfg_k/\|\bfg_k\|_{2}\,\|_{2} \leq \eta$ whenever $\|\Delta_k\|_2 \geq \alpha\rho^{\star}$. Then for any $K$ satisfying:
$$
\RV{K \geq \frac{\left(\Delta_0 - \alpha\eta\right)^{3}}{\left(\alpha\rho^{\star}\right)^{2}\left(\frac{\alpha\nu}{2L} - \alpha\eta\right)}},
$$
Algorithm~\ref{algorithm:INGD} with inputs $x_0,\alpha,K$ and $\eta$ returns  $\bfx_{K}$ satisfying $f(\bfx_K) - f^{\star} \leq \frac{L}{2}\alpha^{2}\left(1 + \rho^{\star}\right)^{2}$.
\end{theorem}

Informally, Theorem \ref{thm:INGD_Convergence} says that if $\alpha = O(\sqrt{\varepsilon})$ then INGD is guaranteed to find $\bfx_{K}$ satisfying $f(\bfx_K) - f^{\star}\leq O(\varepsilon)$ in $O(\varepsilon^{-3/2})$ iterations. Note that if $\eta = 0$ then $\rho^{\star}$ is twice the condition number, $L/\nu$. This theorem extends earlier work of \citep{levy2016power} in two ways:
\begin{enumerate}
    \item Theorem~\ref{thm:INGD_Convergence} allows for errors in the estimates of the normalized gradients.
    \item Theorem~\ref{thm:INGD_Convergence} relaxes the strong convexity assumption to a restricted-strong convexity assumption.
\end{enumerate}
The proof and supporting lemmas are contained in Appendix~\ref{app:Proofs_INGD}. We highlight a curious feature of NGD: in order to achieve an accurate solution one needs to choose a small step-size. In general this cannot be avoided, although we refer to \citep{yu2017normalized} for some ideas on adaptively choosing $\alpha$ if one has access to $\|\bfg_k\|_2$. In Section~\ref{section:Line_search} we discuss how to incorporate a {\em line search} that allows one to use larger step sizes.

\section{The proposed algorithm}
\label{sec:SCOBO}
By combining INGD with 1BitGradEst, we arrive at our proposed algorithm, presented as Algorithm~\ref{algorithm:SCOBO}. The following theorem states our main results precisely.

\begin{algorithm}
\caption{\underline{S}parsity-aware \underline{Co}mparison-\underline{B}ased \underline{O}ptimization (SCOBO)}
\begin{algorithmic}[1]
    \STATE {\bf Inputs:} $\bfx_0,s,m,r$ and $K$
    \FOR{ $k=0,\ldots, K-1$ }
        \STATE $\hat{\bfg}_k \gets \textnormal{1BitGradEst}(\bfx_k,s,m,r)$
        \STATE \RV{Obtain $\alpha_k$ via line search (see Section~\ref{section:Line_search}), or use predetermined $\alpha_k$} 
        \STATE $\bfx_{k+1} = \bfx_{k} - \alpha_k \hat{\bfg}_k$
    \ENDFOR
    \STATE {\bf Output:} $\bfx_{K}$
\end{algorithmic}
\label{algorithm:SCOBO}
\end{algorithm}

\begin{theorem}
\label{thm:FormalMainTheorem_2}
Suppose $f \in\mathcal{F}_{L,\nu,d}$, $f\in CG_{s,d}$ and the comparison oracle $\calC_{f}(\cdot,\cdot)$ satisfies Assumption~\ref{assumption:Oracle}. Choose any step size $\alpha > 0$, target gradient accuracy, $0< \eta < \nu/(2L)$ and let $\Delta_0$ and $\rho^{\star}$ be as defined in Theorem~\ref{thm:INGD_Convergence}. Choose $m,r$ and $K$ according to:
\begin{align*}
K & = \frac{\left(\Delta_0 - \alpha\eta\right)^{3}}{\left(\alpha\rho^{\star}\right)^{2}\left(\frac{\alpha\nu}{2L} - \alpha\eta\right)},  \\
m & = \frac{C}{\RV{\eta^{4}}\delta_0^{2}}s\log(2d/s), \\
r & = \frac{\alpha\RV{\nu}\rho^{\star}}{\RV{2}L\sqrt{d}}.
\end{align*}
Then SCOBO (Algorithm~\ref{algorithm:SCOBO}) with inputs $\bfx_0,s,m,r,K$ and fixed step size $\alpha_k = \alpha$ returns $\bfx_{K}$ satisfying $f(\bfx_{K}) - f^{\star}\leq \frac{L}{2}\alpha^{2}(1+\rho^{\star})^{2}$ using $mK$ oracle queries, with probability at least
$1 - 8K\exp\left(-c\eta^{4}s\log(2d/s)\right)$.
\end{theorem}
Informally stated, this theorem shows that if $\alpha = O(\sqrt{\varepsilon})$ then SCOBO will find $\bfx_{K}$ satisfying $f(\bfx_K) - f^{\star} \leq \varepsilon$ using only $\tilde{O}\left(s\varepsilon^{-3/2}\delta_{0}^{-2}\right)$ queries. 
The proof of Theorem~\ref{thm:FormalMainTheorem_2} is in Section~\ref{sec:SCOBO_Proofs}. We also highlight the following consequence of Theorem~\ref{thm:FormalMainTheorem_2}:

\begin{corollary}
In addition to the  assumptions in Theorem~\ref{thm:FormalMainTheorem_2} \RV{suppose $\nu/2L \leq 0.5$ and take $\alpha_k = \alpha = (1+\rho^{\star})\sqrt{2\varepsilon/L}$}, let $\bfx_1,\bfx_2,\ldots$ be the iterates produced by SCOBO. Then with probability at least $1 - 8K\exp\left(-c\eta^{4}s\log(2d/s)\right)$ either:
\begin{enumerate}
    \item $f(\bfx_{k+1}) \leq f(\bfx_k)$, or
    \item $f(\bfx_{k}) \leq \varepsilon$
\end{enumerate}
holds for all $0\leq k \leq K-1$.
\label{lemma:SCOBO_Guaranteed_Descent}
\end{corollary}
In other words, with overwhelming probability, SCOBO is a descent algorithm until it hits the target accuracy. We verify this experimentally in Section~\ref{section:NumericalExperiments}. 

\begin{remark}
\RV{
As observed elsewhere \citep{golovin2019gradientless,matsui2017parallel}, one can easily extend Theorem~\ref{thm:FormalMainTheorem_2} to compositions of functions.
That is, Theorem~\ref{thm:FormalMainTheorem_2} holds as stated if we instead assume that $f(\bfx) = g(h(\bfx))$ with $h\in \mathcal{F}_{L,\nu,d}$, $h\in CG_{s,d}$ and $g:\mathbb{R}\to\mathbb{R}$ {\em any} monotonically increasing function.  }
\end{remark}

\section{Line search}
\label{section:Line_search}
\begin{algorithm}[h]
\caption{Inexact line search for SCOBO} \label{algorithm:Line Search}
\begin{algorithmic}[1]
\STATE {\bfseries Input:} $\bfx$: current point; ${\hat{\bfg}}_{k}$: estimated gradient; $\alpha_{\textrm{def}}$: default step size; $M$: number of trials for comparison; $\omega \geq 0$: confidence parameter; $\psi > 1$: searching parameter.
\STATE $\alpha = \alpha_{\textrm{def}}$
    \WHILE{$\calC_f^M(\bfx+\alpha {\hat{\bfg}}_{k}, \bfx + \psi \alpha {\hat{\bfg}}_{k} )\leq-\omega$}
    \STATE $\alpha= \psi\alpha$
    \ENDWHILE
\STATE {\bfseries Output:} $\alpha$
\end{algorithmic}
\end{algorithm}


\RV{As mentioned in Section~\ref{section:INGD}, by the nature of the comparison oracle, the length of the true gradient is not recoverable.} By Theorem~\ref{thm:FormalMainTheorem_2}, we can guarantee the convergence of SCOBO with a fixed small step size; however, it appears that using longer step sizes may significantly accelerate the convergence, particularly in the earlier stages of SCOBO where the length of the true gradient is larger. Hence, we propose an inexact step size line search method, summarized as Algorithm~\ref{algorithm:Line Search}. 

The main challenge for our line search is the noisy comparison oracle. To overcome this, first define the $M$-trial comparison oracle:
\begin{align} \label{eq:m-trial}
    \calC_f^M(\bfx, \bfy) = \Big(\sum\nolimits_{i=1}^M \calC_f (\bfx, \bfy)|_{i\textnormal{-th query}}\Big) \Big/ M.
\end{align}
When $M$ is large enough, we will have $\sign(\calC_f^M(\bfx, \bfy))=\sign(f(\bfy)-f(\bfx)) $ with high probability. In particular, when $\kappa=1$ and $f(\bfy)<f(\bfx)$, take $M=\beta\delta_0^{-2}$, then $\calC_f^M(\bfx, \bfy) <- \delta_0$ with probability at least $1-\exp(-\beta/2)$.  
If $\kappa>1$ the probability that $C_{f}(\bfx,\bfy) = \sign(f(\bfy) - f(\bfx))$ depends on $|f(\bfy)-f(\bfx)|$, which means the theoretical $M$ required cannot be computed {\em a priori} unless we in addition assume strong convexity as in \citep{Jamieson2012}. In practice, we pick a fixed $M$ and assign a confidence parameter $\omega\geq0$ so that $f(\bfy)<f(\bfx)$ with high probability when $\calC_f^M(\bfx, \bfy )\leq-\omega$. Starting with an initial step size $\alpha_{\textrm{def}}$, the line search algorithm will repeatedly increase the step size by some factor $\psi>1$ until $f(\bfx+\alpha {\hat{\bfg}}_{k}) \geq f(\bfx + \psi \alpha {\hat{\bfg}}_{k} )$. When Algorithm~\ref{algorithm:Line Search} stops, the output $\alpha$ is unlikely optimal; however, with high probability, it satisfies 
\begin{equation*}
\psi^{-1}\alpha_\star<\alpha\leq\alpha_\star
\end{equation*}
where $\alpha_\star=\argmin_{\alpha} f(\bfx+\alpha {\hat{\bfg}}_k)$. Since the estimated gradient is close to the normalized true gradient, we conclude $\alpha_\star\geq c_{\alpha}\|{{\bfg}}_k \|_2/L$ where $c_\alpha$ is a constant depending on $\|\hat{\bfg}_k-\frac{\bfg_k}{\|\bfg_k\|_2}\|_2$. Therefore, $\alpha\in (\psi^{-1}c_\alpha\|{{\bfg}}_k\|_2/L, \alpha_\star]$ is a reasonably good step size. 
If one wishes to estimate $\alpha_\star$ more accurately, we can further apply Fibonacci search on the interval $[\alpha, \psi \alpha]$. Either way, the query complexity of the inexact line search is $\mathcal{O}(M\log_\psi(\alpha_\star/\alpha_{\mathrm{def}}))$.


\subsection{Warm started line search} 

\begin{algorithm}[h]
\caption{Warm started inexact line search}
\begin{algorithmic}[1]
\STATE {\bfseries Input:} $\bfx$: current point; ${\hat{\bfg}}_{k}$: estimated gradient; $\alpha$: initial step size; $\alpha_{\textrm{def}}$: default step size; $M$: number of trials for comparison; $\omega$: confidence parameter; $\psi$: searching parameter.
\IF{$\calC_f^M(\bfx, \bfx+\alpha {\hat{\bfg}}_{k} )\leq-\omega$}
    \WHILE{$\calC_f^M(\bfx+\alpha {\hat{\bfg}}_{k}, \bfx + \psi \alpha {\hat{\bfg}}_{k} )\leq-\omega$}
    \STATE $\alpha= \psi\alpha$
    \ENDWHILE
\ELSIF{$\calC_f^M(\bfx, \bfx+\alpha {\hat{\bfg}}_{k} ) \geq \omega$}
    \WHILE{$\calC_f^M(\bfx, \bfx + \psi^{-1} \alpha {\hat{\bfg}}_{k} )\geq\omega$ \AND $\alpha>\alpha_{\textrm{def}}$}
    \IF{$\psi^{-1}\alpha<\alpha_{\textrm{def}}$}
        \STATE $\alpha=\alpha_{\textrm{def}}$
    \ELSE
        \STATE $\alpha= \psi^{-1}\alpha$
    \ENDIF
    \ENDWHILE
\ENDIF
\STATE {\bfseries Output:} $\alpha$
\end{algorithmic}
\label{algorithm:Warm Started Line Search}
\end{algorithm}

We introduce a warm started inexact line search method for SCOBO. The vanilla inexact line search, {\em i.e.} Algorithm~\ref{algorithm:Line Search}, starts its step size searching from $\alpha = \alpha_{\mathrm{def}}$ at every iteration of SCOBO. Although $\alpha$ converges to the interval $(\alpha_\star/2,\alpha_\star]$ exponentially, it may still waste unnecessary effort in the case that optimal step sizes do not change  much between consecutive iterations. Especially, when $\alpha_\star/\alpha_{\mathrm{def}}$ is larger, a noticeable difference, in terms of the number of comparison oracle queries, can be observed if we do not restart the line search all over every iteration. 

In the warm started inexact line search, we use the estimated step size from the last iteration of SCOBO as the initialization for the new line search. Since the warm started initialization can be larger than the optimal step size at the current iteration, we must also include a mechanism to reduce step size from the initial $\alpha$. We first use the $M$-trial comparison oracle to determine if we want to extend or reduce the initial step size with confidence. We will keep the initial step size if the confidence is mediocre in both directions.  Once decided, we keep extending/reducing the step size by a factor of $\psi$ until the stopping condition is satisfied. 

By Theorem~\ref{thm:FormalMainTheorem_2}, we can use a smaller minimum step size for targeting a better error bound, so the final convergence accuracy of SCOBO with warm started line search is better than SCOBO with bigger fixed step size. Though SCOBO with vanilla line search can achieve similar accuracy by setting a very small default step size, it can waste many queries on line search if the default step size is too tiny. However, if the default step size is too large, then the accuracy of SCOBO become loose. 
Overall, we claim the warm started line search with small default step size has both good convergence performance and query efficiency. We summarize the warm started inexact line search as Algorithm~\ref{algorithm:Warm Started Line Search}.

\subsection{Choosing the step size}  We have experimented with constant step size, line search and decaying step sizes. SCOBO works well in all three cases. Line search provides faster convergence for convex functions (Section~\ref{sec:Synthetic}) while using decaying step sizes provides more stable performance for highly non-convex functions such as in the MuJoCo control problems (Section~\ref{sec:Mujoco}). If ensuring descent ({\em i.e.} $f(\bfx_{k+1}) \leq f(\bfx_k)$) at every step is crucial, we recommend using a fixed small step size.

\section{Numerical experiments} \label{section:NumericalExperiments}

In this section, we demonstrate the empirical performance of SCOBO on both synthetic examples and the MuJoCo dataset \citep{todorov2012mujoco}. The codes for SCOBO can be found online:
\begin{center}
    \url{https://github.com/caesarcai/SCOBO}.
\end{center}



\subsection{Synthetic examples}  \label{sec:Synthetic}

We benchmark SCOBO on four synthetic test cases: 
\begin{enumerate}[label={(\alph*)}]
    \item \label{case:a}We consider the skewed-quartic function used in \citep{spall2000adaptive}. We embed the $20$-dimension skewed-quartic function into $500$-dimensional space.
    The comparison oracle parameters are set to be $\kappa=1.5$, $\mu=1$ and $\delta_0=0.5$.
    \item \label{case:b}We consider the squared sum of the $20$ largest-in-magnitude elements in a $500$-dimensional vector, {\em i.e.} $f(\bfx)=\sum_{\RV{i=1}}^{20} x_{m_i}^2$ where $x_{m_i}$ is the $i$-th largest-in-magnitude entry. The comparison oracle parameters are set to be $\kappa=1.5$, $\mu=4$ and $\delta_0=0.5$.
    \item \label{case:c}We use the same objective function as in case~\ref{case:a}, but the comparison oracle parameters are set to be $\kappa=1$, $\mu=1$ and $\delta_0=0.3$. 
    \item \label{case:d}We use the same objective function as in case~\ref{case:b}, but the comparison oracle parameters are set to be $\kappa=1$, $\mu=1$ and $\delta_0=0.3$.
\end{enumerate}

All four test cases have $s=20$ and $d=500$. \RV{When $\calC_{f}(\cdot,\cdot)$ satisfies Assumption~\ref{assumption:Oracle}, ({\em i.e.} cases (c) and (d)) Theorem~\ref{thm:FormalMainTheorem_2} requires $m \approx \delta_0^{-2}s\log(2d/s) \approx 11s\log(2d/s)$, so we play it safe and choose $m=20s\log(2d/s)$. We use the same value of $m$ in cases (a) and (b).} In cases~\ref{case:a} and \ref{case:b}, the flipping probability of $\calC_f(\bfx,\bfy)$ will rise when $|f(\bfx)-f(\bfy)|$ is small, so we set a fixed sampling radius $r=1/2\sqrt{s}$ in these two cases. In contrast, the comparison oracle parameters in cases~\ref{case:c} and \ref{case:d} imply  $\mathbb{P}\left[\calC_{f}(\bfx,\bfy) = \sign\left(f(\bfy) - f(\bfx)\right)\right] = 0.8$, so the flipping probability of $\calC_{f}(\bfx,\bfy)$ is independent of $|f(\bfy)-f(\bfx)|$. Thus, we may use a smaller sampling radius of $r=10^{-4}$, which offsets the perturbation due to $\nabla^2 f(\bfx)$ (see \eqref{eq:Taylor_Expansion}).


\begin{figure}[t]
\begin{center}
\centering
\hfill
\includegraphics[width = .45\linewidth, height = .31\linewidth]{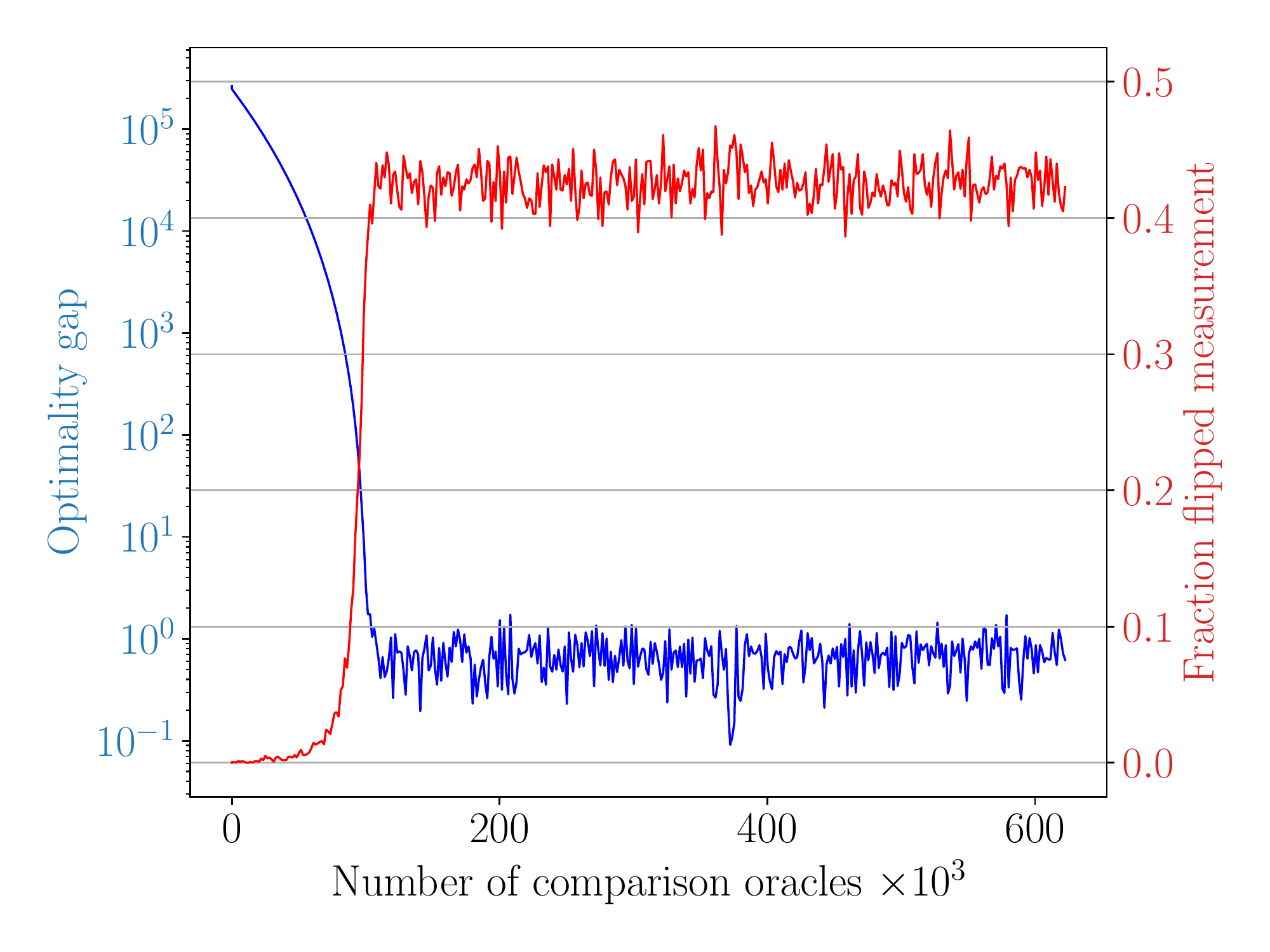}
\hfill
\includegraphics[width = .45\linewidth, height = .31\linewidth]{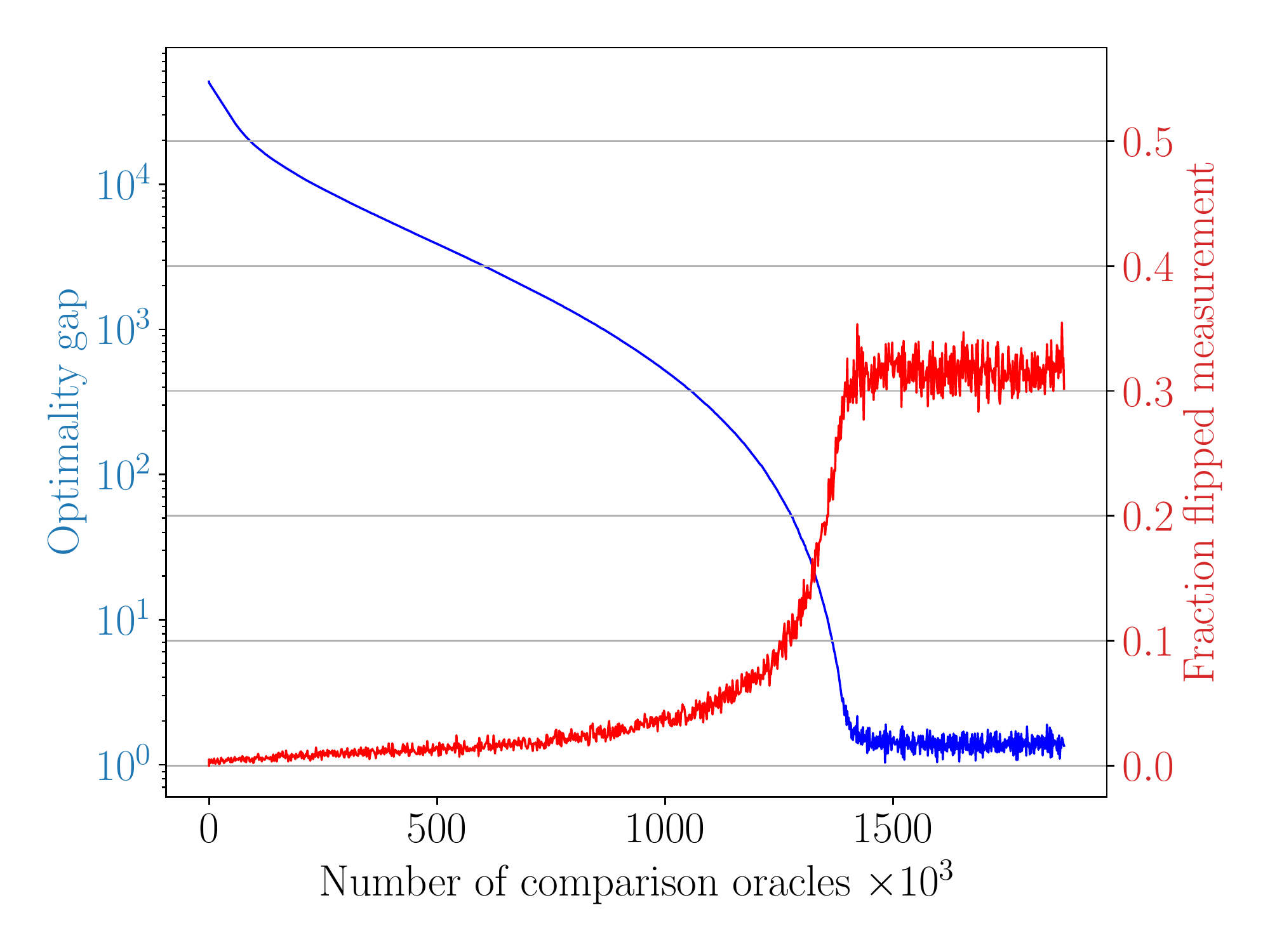} 
\hfill\\
\hfill
\includegraphics[width = .45\linewidth, height = .31\linewidth]{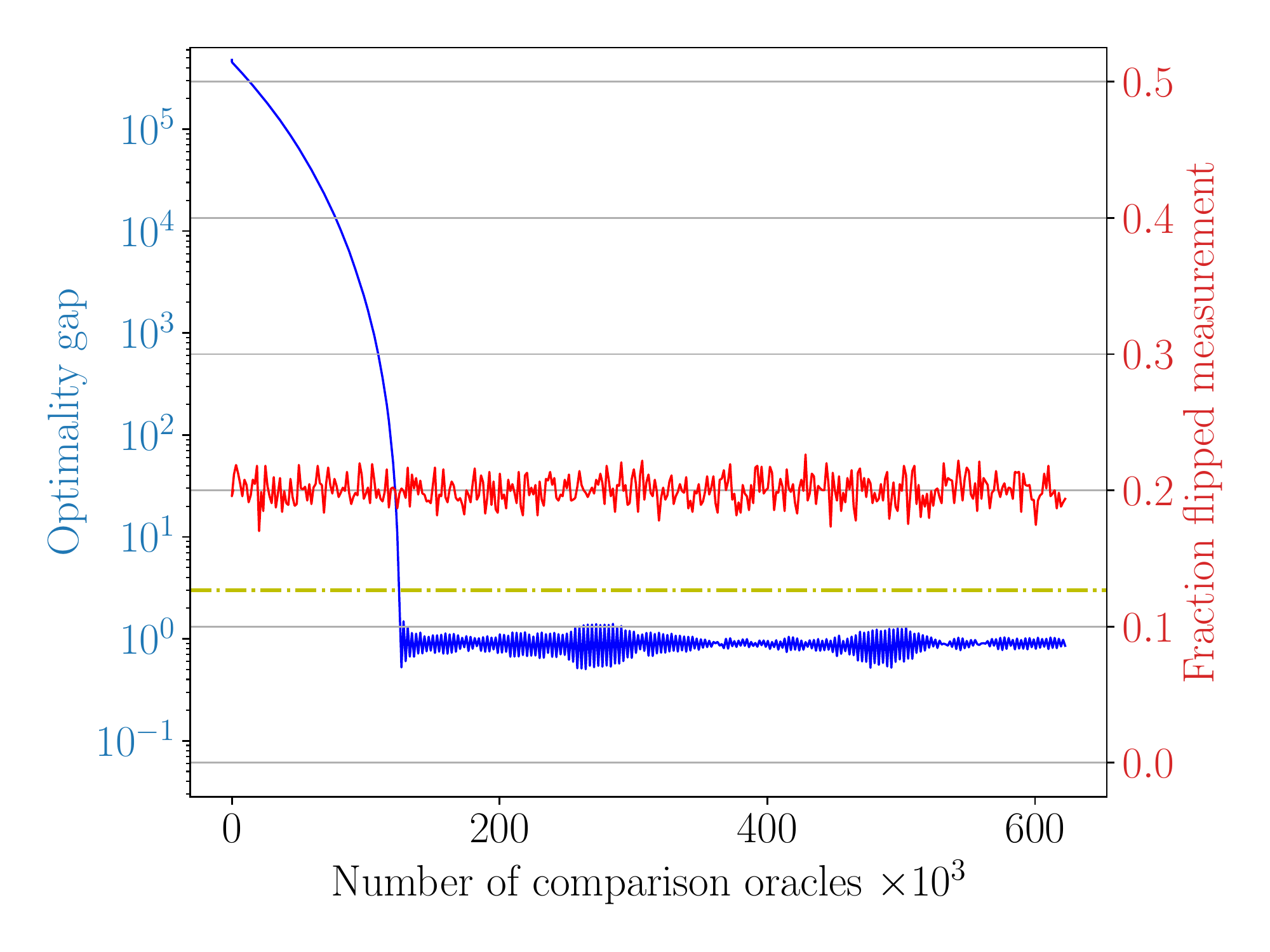}
\hfill
\includegraphics[width = .45\linewidth, height = .31\linewidth]{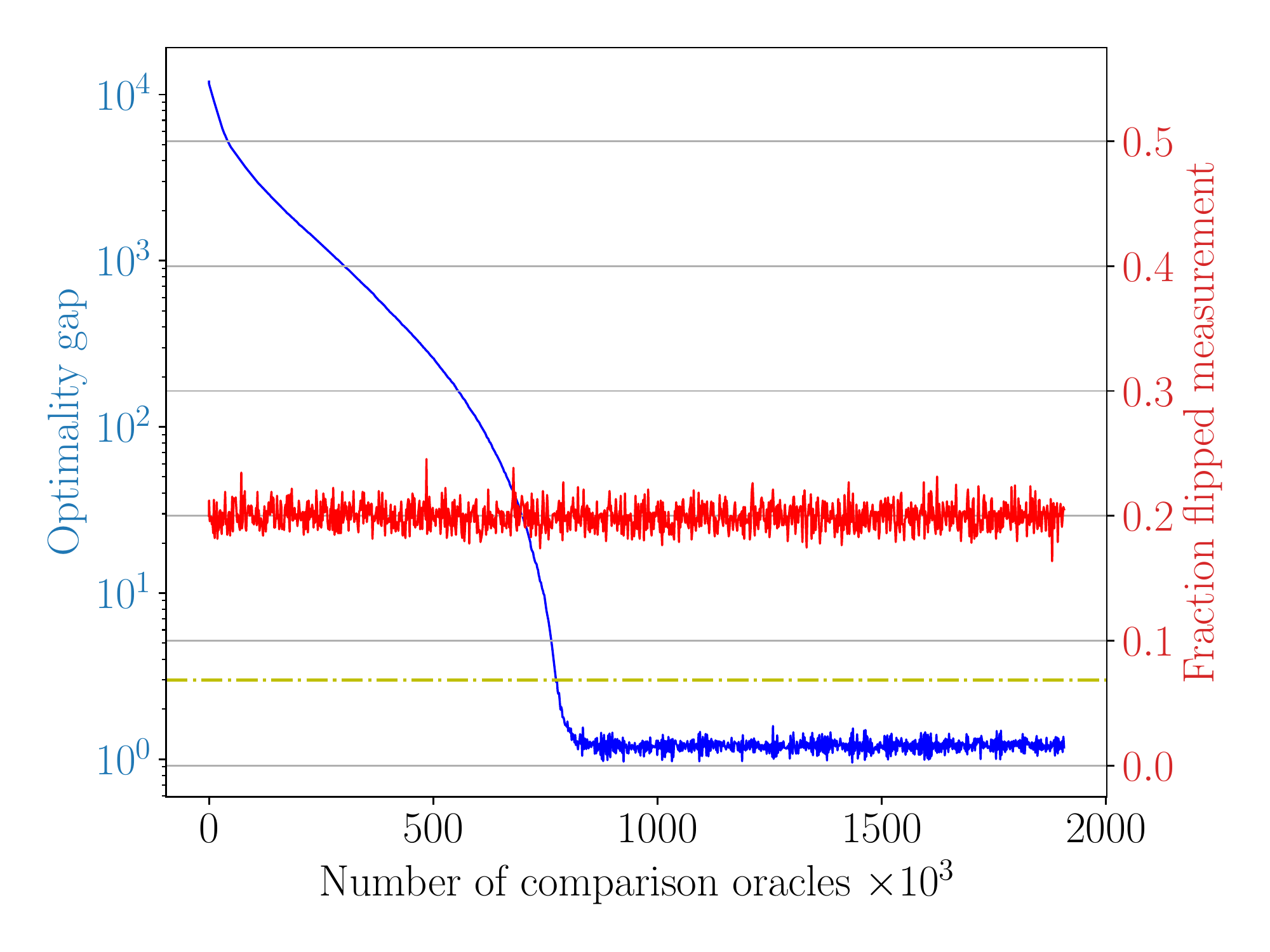}
\hfill
\end{center}
\vspace{-0.05in}
\caption{ Optimality gap, fraction of flipped comparison oracles {\em v.s.} number of comparison oracles used. {\bf Top-left:} case~\ref{case:a}. {\bf Top-right:} case~\ref{case:b}. {\bf Bottom-left:} case~\ref{case:c}. {\bf Bottom-right:} case~\ref{case:d}.}
\label{figure:Synthetic flipping}
\end{figure}

We first numerically verify the convergence of SCOBO with fixed step size $\alpha = 2$. In Figure~\ref{figure:Synthetic flipping}, we plot the convergence trajectory of SCOBO in blue for each of the four cases. For comparison, we also plot the fraction of flipped oracle queries in red. We further plot the theoretical error bound as a horizontal yellow dash line for reference in cases~\ref{case:c} and \ref{case:d}. 

In cases~\ref{case:a} and \ref{case:b}, SCOBO converges slowly yet smoothly; meanwhile, the fraction of flipped comparison oracle queries keeps relatively low in the early stage. The number of flipped comparison increase rapidly as the optimality gap gets smaller. While the expectation of flipping probability later rises to over $40\%$ and $30\%$ respectively, SCOBO stays stably near the optimum. 

In cases~\ref{case:c} and \ref{case:d}, the fraction of flipped oracle queries is constantly $20\%$ in expectation. This may seem to create a harder 1-bit compressed sensing problem for Algorithm \ref{alg:1BitGradEstimator}, but this difficulty is offset by the smaller sampling radius. Hence, the trajectory of SCOBO shows smooth monotonic descent to the theoretical bound in both cases. 

Overall, we observe that SCOBO converges successfully 
in all cases, and then remains near the optimum.
This verifies our convergence theorem, {\em i.e.} Theorem~\ref{thm:FormalMainTheorem_2}.

\begin{figure}[t]
\begin{center}
\centering
\hfill
\includegraphics[width = .45\linewidth, height = .31\linewidth]{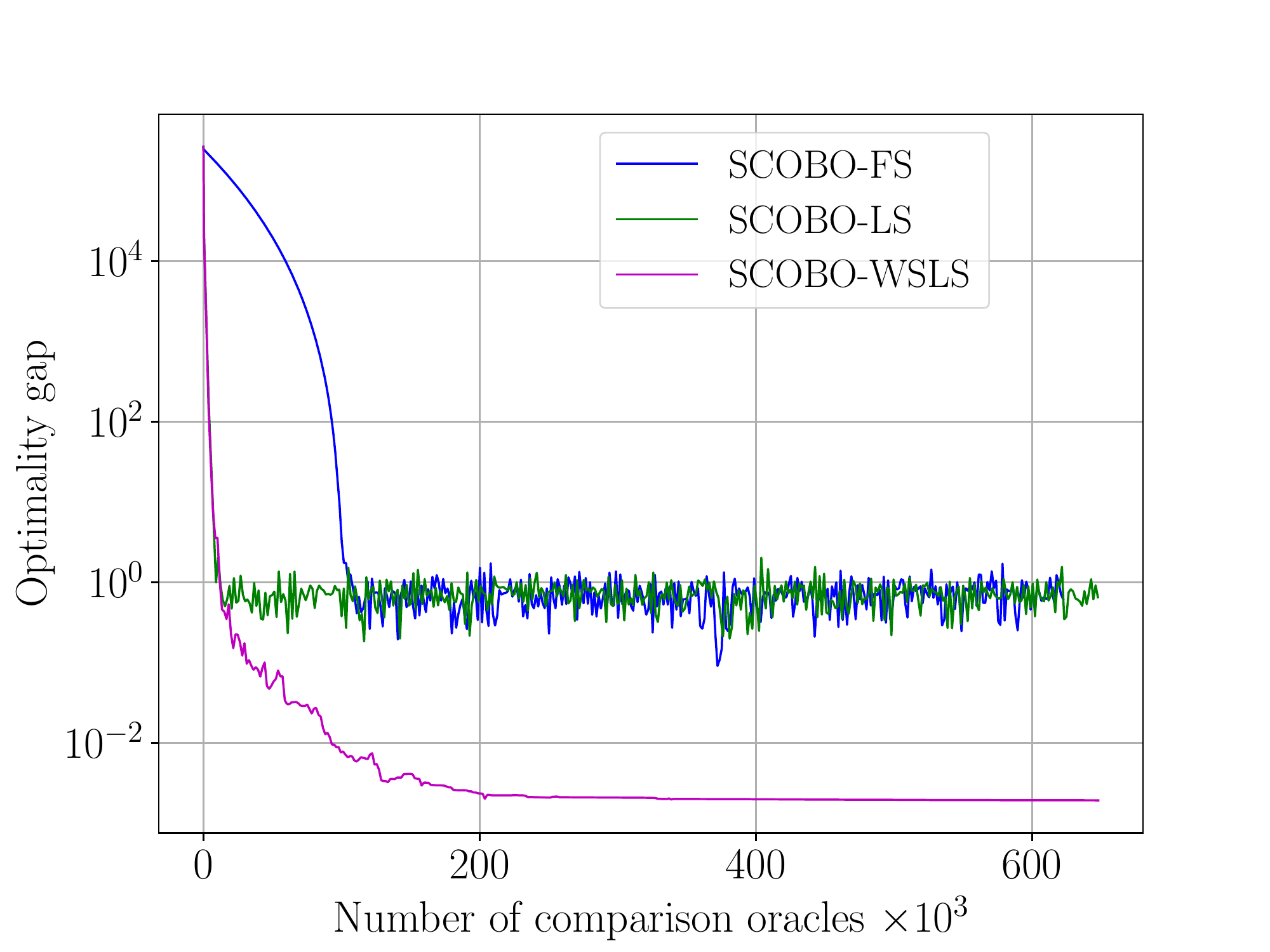} 
\hfill
\includegraphics[width = .45\linewidth, height = .31\linewidth]{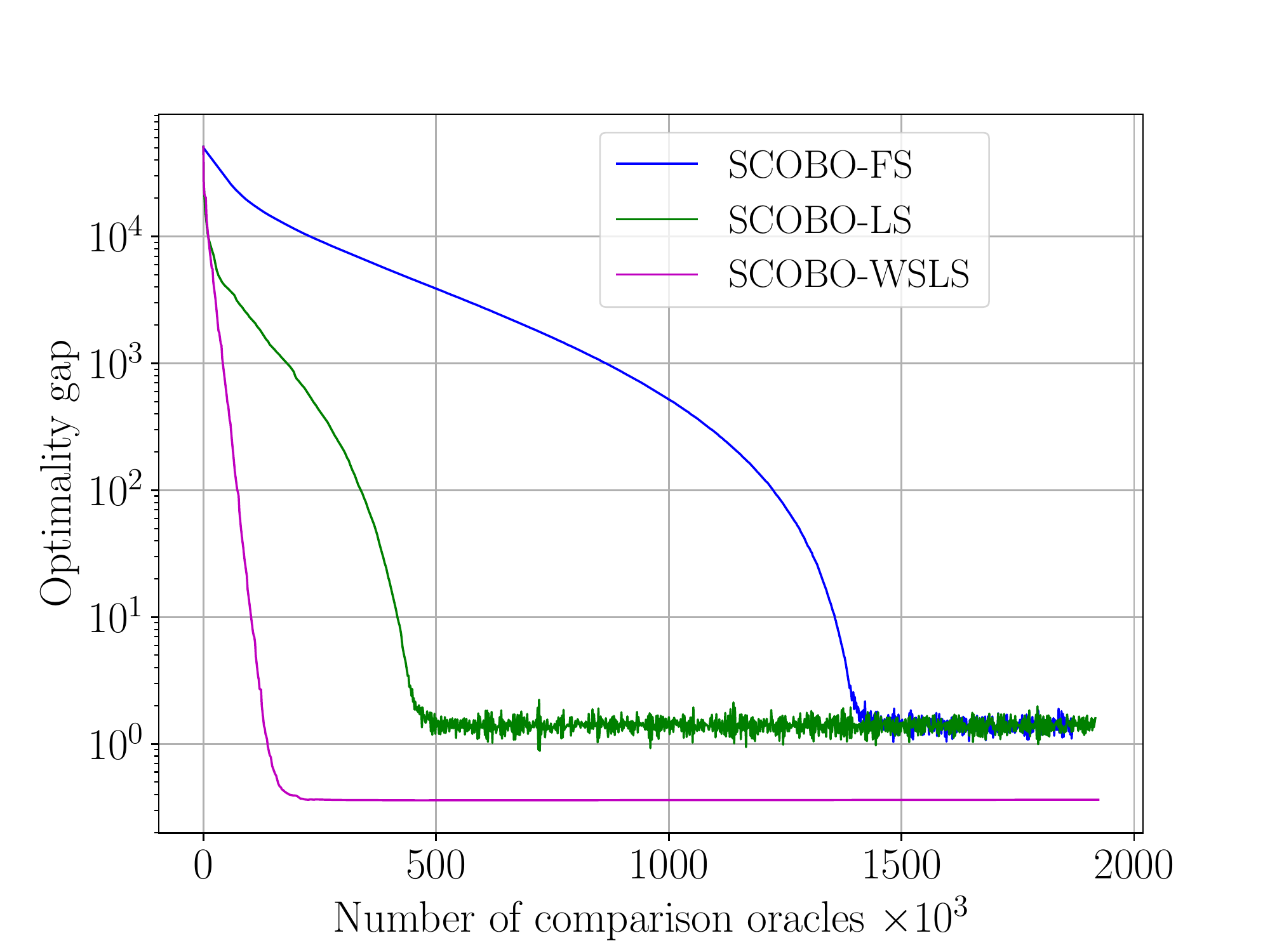} \hfill\\
\hfill
\includegraphics[width = .45\linewidth, height = .31\linewidth]{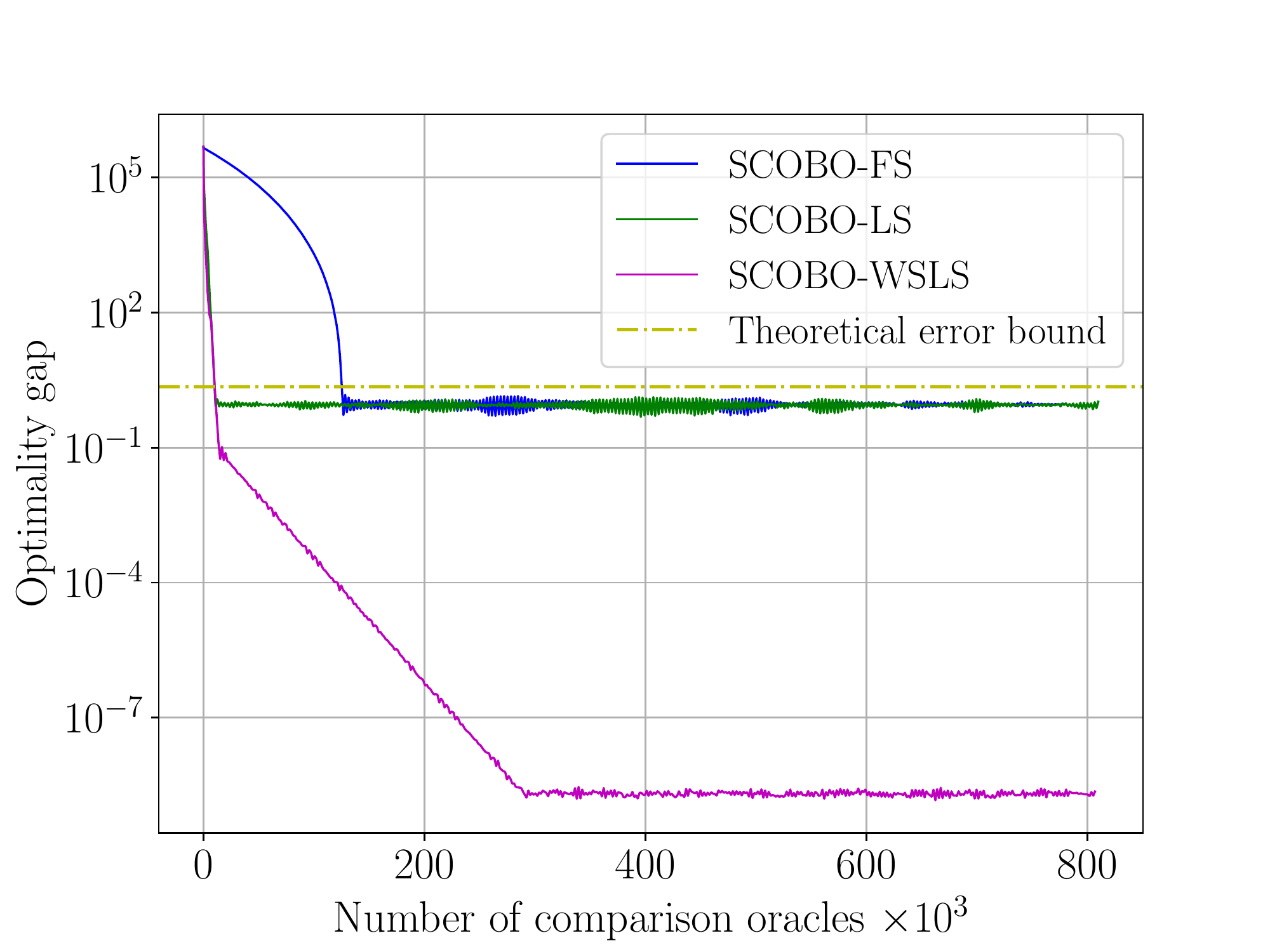}\hfill
\includegraphics[width = .45\linewidth, height = .31\linewidth]{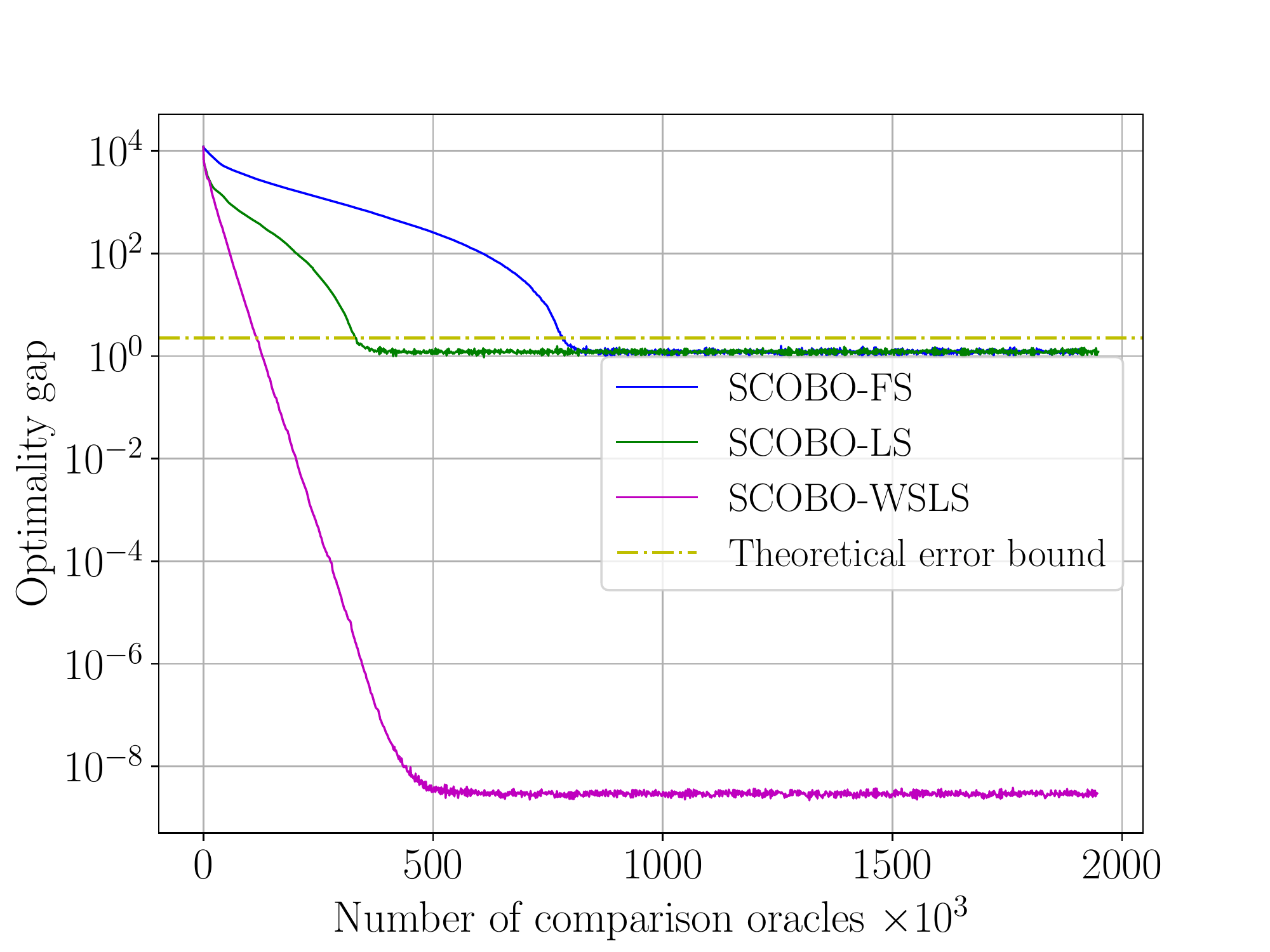}
\hfill
\end{center}
\vspace{-0.05in}
\caption{Convergence comparison between fixed step size version and line search version of SCOBO. {\bf Top-left:} case~\ref{case:a}. {\bf Top-right:} case~\ref{case:b}. {\bf Bottom-left:} case~\ref{case:c}. {\bf Bottom-right:} case~\ref{case:d}.}
\label{figure:Synthetic linesearch}
\end{figure}

We investigate the empirical performance of different versions of SCOBO: fixed step size (SCOBO-FS), vanilla line search (SCOBO-LS) and warm started line search (SCOBO-WSLS). For all test cases, we use the default step size $\alpha_{\mathrm{def}}=2$ for SCOBO-LS and $\alpha_{\mathrm{def}}=10^{-4}$ for SCOBO-WSLS.
The line search parameters are set to be $M=40$, $\omega=0.05$ and $\psi=2$. The results are shown in Figure~\ref{figure:Synthetic linesearch}. We find both versions of inexact line search methods accelerate the convergence dramatically in all test cases. Furthermore, we see SCOBO-WSLS is able to \RV{converge} to higher accuracy since it can use a tiny default step size without wasting unnecessary queries on distant cold started line search. In summary, SCOBO can be stably accelerated with the proposed line search methods. 

Finally, we compare SCOBO against two state-of-the-art comparison oracle based optimization methods: Pairwise comparison coordinate descent (PCCD) \citep{Jamieson2012}, and SignOPT \citep{cheng2019sign}. We implemented PCCD by ourselves and hand tuned its parameters for the best performance. The code for SignOPT is obtained from the authors' website, we use the parameters suggested in the paper; in particular, we sample $200$ random directions for their gradient estimator, which is recommend by the authors. We also emphasize that we use same key parameters ({\em e.g.} sampling radius) for all three tested algorithms, so we do not gain advantage from the parameter setting.
The empirical results are summarized in Figure~\ref{figure:Synthetic comparsion}. 

\begin{figure}[t]
\begin{center}
\centering
\hfill
\includegraphics[width = .45\linewidth, height = .31\linewidth]{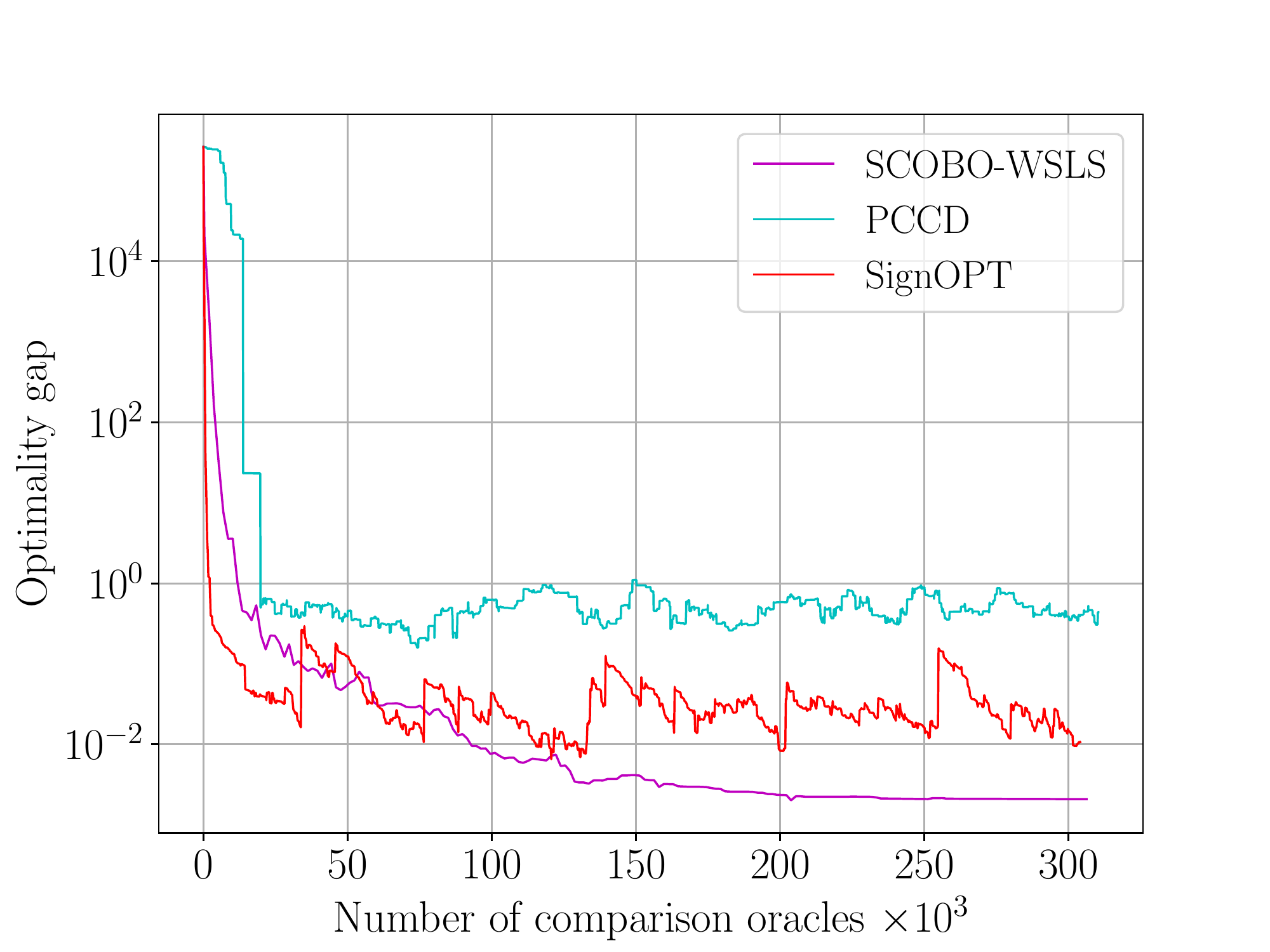} 
\hfill
\includegraphics[width = .45\linewidth, height = .31\linewidth]{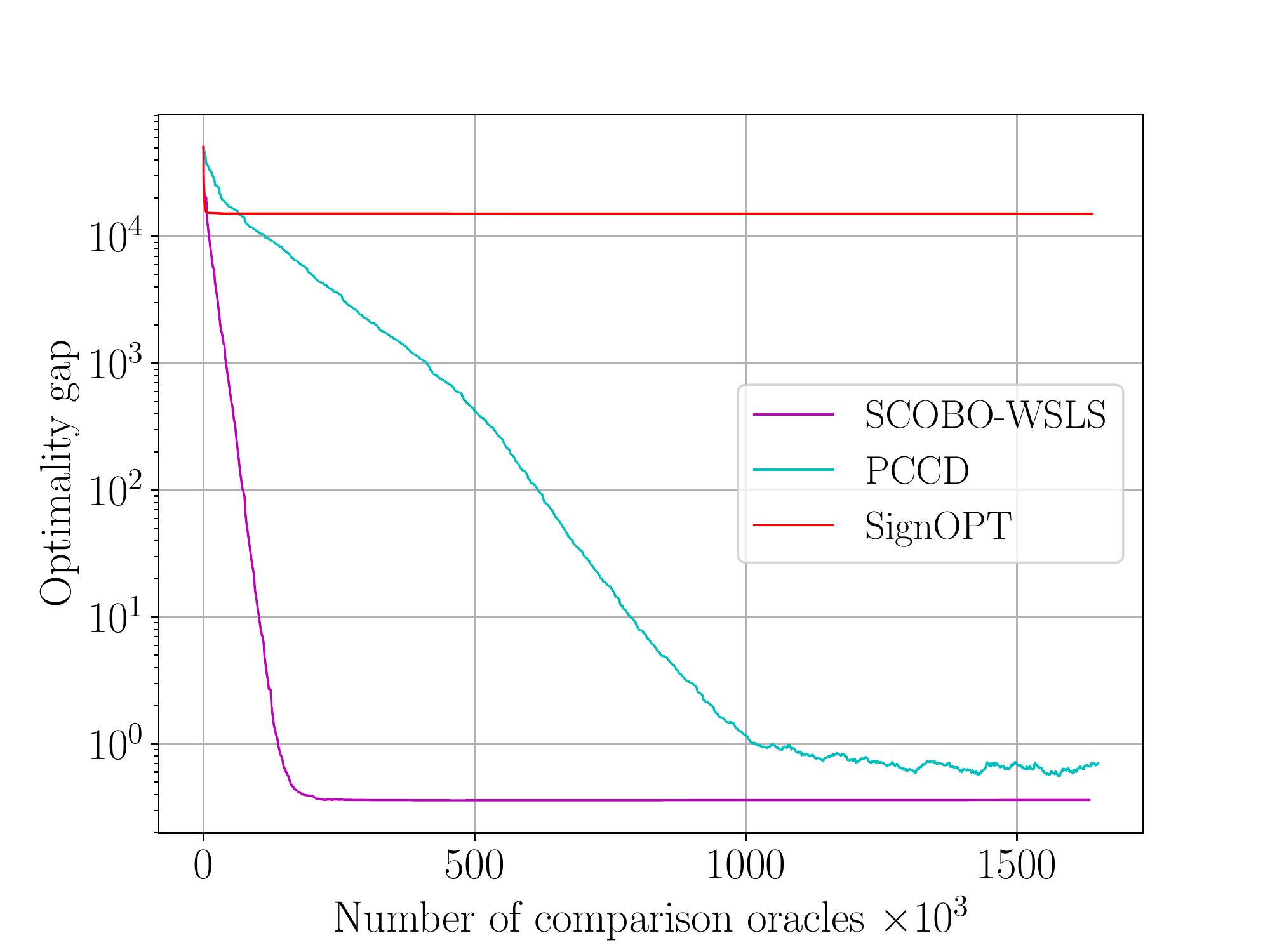} 
\hfill\\
\hfill
\includegraphics[width = .45\linewidth, height = .31\linewidth]{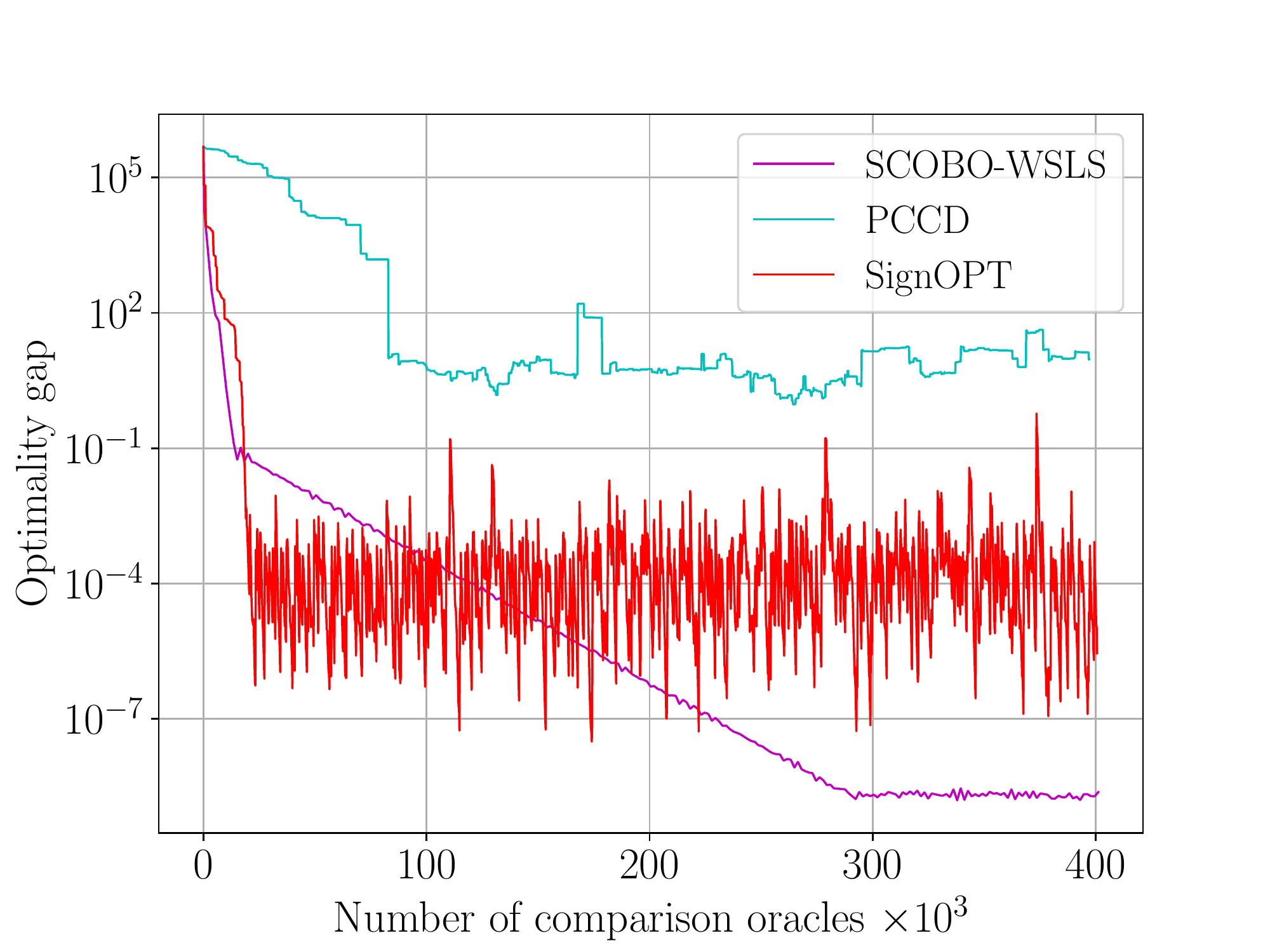}\hfill
\includegraphics[width = .45\linewidth, height = .31\linewidth]{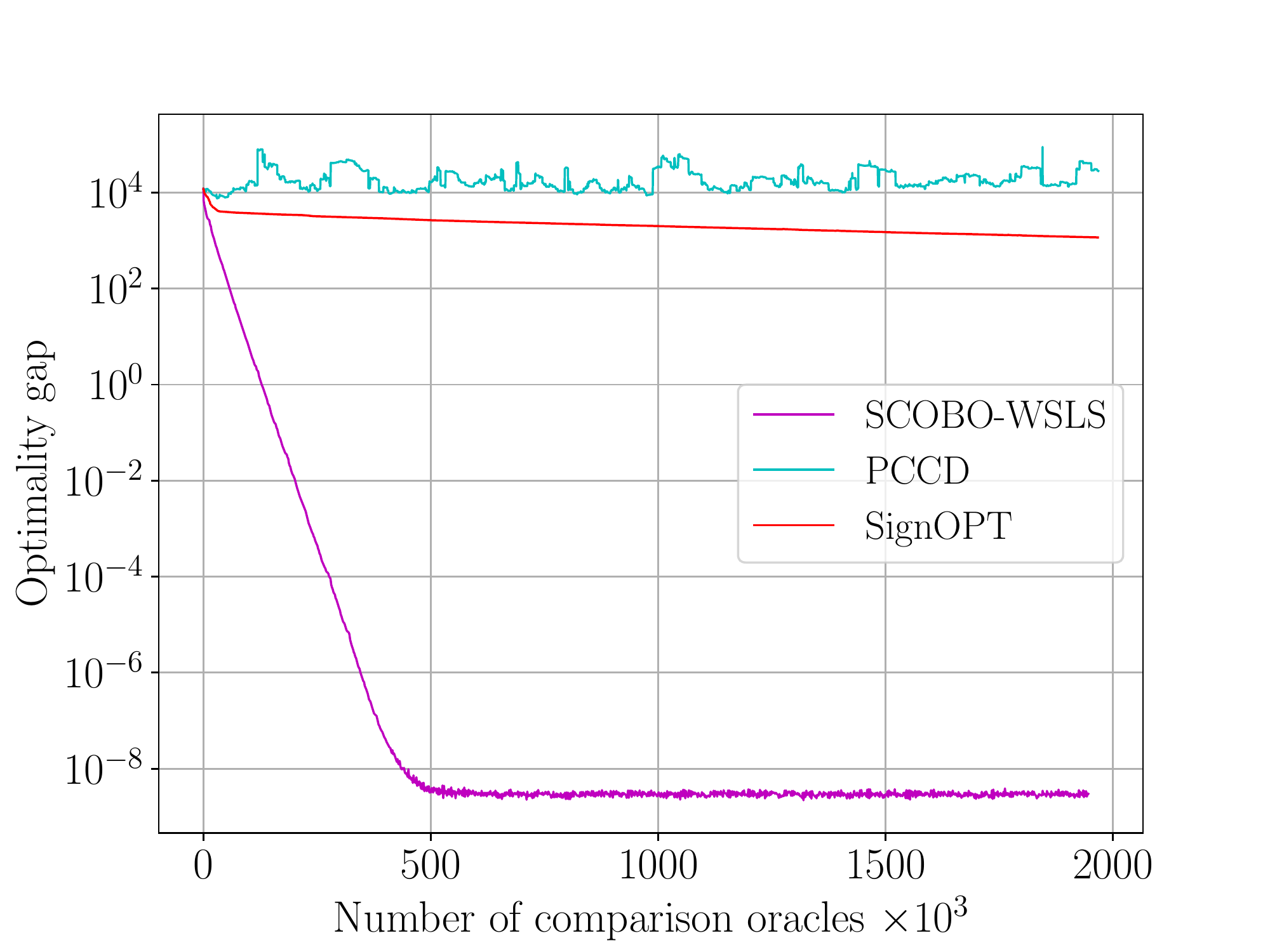}
\hfill
\end{center}
\vspace{-0.08in}
\caption{Convergence comparison among SCOBO, PCCD and SignOPT. {\bf Top-left:} case~\ref{case:a}. {\bf Top-right:} case~\ref{case:b}. {\bf Bottom-left:} case~\ref{case:c}. {\bf Bottom-right:} case~\ref{case:d}.}
\label{figure:Synthetic comparsion}
\vspace{-0.05in}
\end{figure}

SignOPT has a slight advantage in the early stage of test cases~\ref{case:a} and \ref{case:c}, but SCOBO stably converges to more accurate solutions. For the harder cases~\ref{case:b} and \ref{case:d}, SignOPT fails. Note that the support of the gradient is fixed in cases~\ref{case:a} and \ref{case:c}  while the gradient support varies in cases~\ref{case:b} and \ref{case:d}. The varying gradient support does not effect SCOBO, but it is problematic for SignOPT.

PCCD has reasonable performance in test cases~\ref{case:a} and \ref{case:b}, but fails cases~\ref{case:c} and \ref{case:d} where $\kappa=1$. This is caused by the fact that PCCD uses a $1$-trial comparison oracle for coordinate line search. When the fraction of flipped queries is constantly high this line search is unreliable. Using the $M$-trial comparison oracle \eqref{eq:m-trial} {\em could} improve the quality of line search, but is unlikely to improve the overall efficiency. This is because PCCD uses an unreliable search direction and spends all its queries on line search. Thus, using \eqref{eq:m-trial} will immediately increase the total queries $M$-fold without necessarily yielding more descent per iteration. In contrast, SCOBO starts with a very good search direction ($\approx \bfg_k$), and thus it makes sense to invest more queries in a more thorough line search. 

In conclusion, we find SCOBO has the best performance among the three tested algorithms, in terms of both query complexity and convergence stability.


\subsection{MuJoCo policy optimization}
\label{sec:Mujoco}

In this section, we use SCOBO to learn a policy for simulated robot control, using only comparison oracle feedback, for several problems from the MuJoCo suite of benchmarks \citep{todorov2012mujoco}. Inspired by \citep{mania2018simple}, we use a simple class of policies (linear policies) and minimal computational resources. We note that the objective functions for these problems ({\em i.e.} the reward obtained given an input policy) are highly non-convex and possess no obvious low dimensional structure. Nevertheless, SCOBO performs well. 
Our experimental setup is as follows:
\begin{itemize}
\item We use a horizon of 1000 iterations for each rollout.
\item The only access to the reward function was through a poynomial comparison oracle with $\kappa = 2,\mu = 0.5$ and $\delta_0 = 0.3$.
\item The values of $m$ and $s$ for each experiment are displayed in Table \ref{table:vals_for_MuJoCo}. Note that these values are somewhat arbitrary; empirically we observed good performance for a broad range of $m$ and $s$ values.
\item We do not use line search. Instead, we implement an exponentially decaying learning rate schedule.
\item We use the state normalization/whitening trick introduced in \citep{mania2018simple} to encourage more equal exploration across dimensions.
\end{itemize}

\begin{table}[t]
\caption{Parameters for SCOBO applied to MuJoCo. Note that for each model, Dim.~is the dimension of the action space times the dimension of the observation space.} \label{table:vals_for_MuJoCo}
\vspace{-0.1in}
\center
\begin{tabular}{c  c c c}
\toprule
    Model & Dim.  & $m$  & $s$ \\
    \midrule
     Swimmer-v2 & $16$  & $10$ & $5$ \\
     \hline
     Reacher-v2 & $22$  & $26$ & $16$ \\
     \hline
     HalfCheetah-v2 & $102$ & $100$ & $50$ \\
     \bottomrule
\end{tabular}
\end{table}

On all our tests, the mean rewards eventually exceed the reward threshold specified in the OpenAI Gym environment. Compared with reinforcement learning and gradient estimation approaches in the literature, SCOBO on Swimmer-v2 and Reacher-v2 yields surprisingly competitive convergence in terms of number of queries required. Note that typical approaches to reinforcement learning receive the reward function value, encoded as a 32-bit float, upon each query. In contrast, SCOBO only receives 1 bit per query. When performance is measured as {\em the number of bits required to exceed the reward threshold}, the performance of SCOBO exceeds that of the state of the art. For example, TD3 and CEM-TD3 \citep{pourchot2018cemrl} require roughly $3.2$ million bits for HalfCheetah-v2, whereas SCOBO requires only around $400$ thousand bits.

\begin{figure}[t]
\centering
\hfill
\subfloat{\includegraphics[width=.333\linewidth]{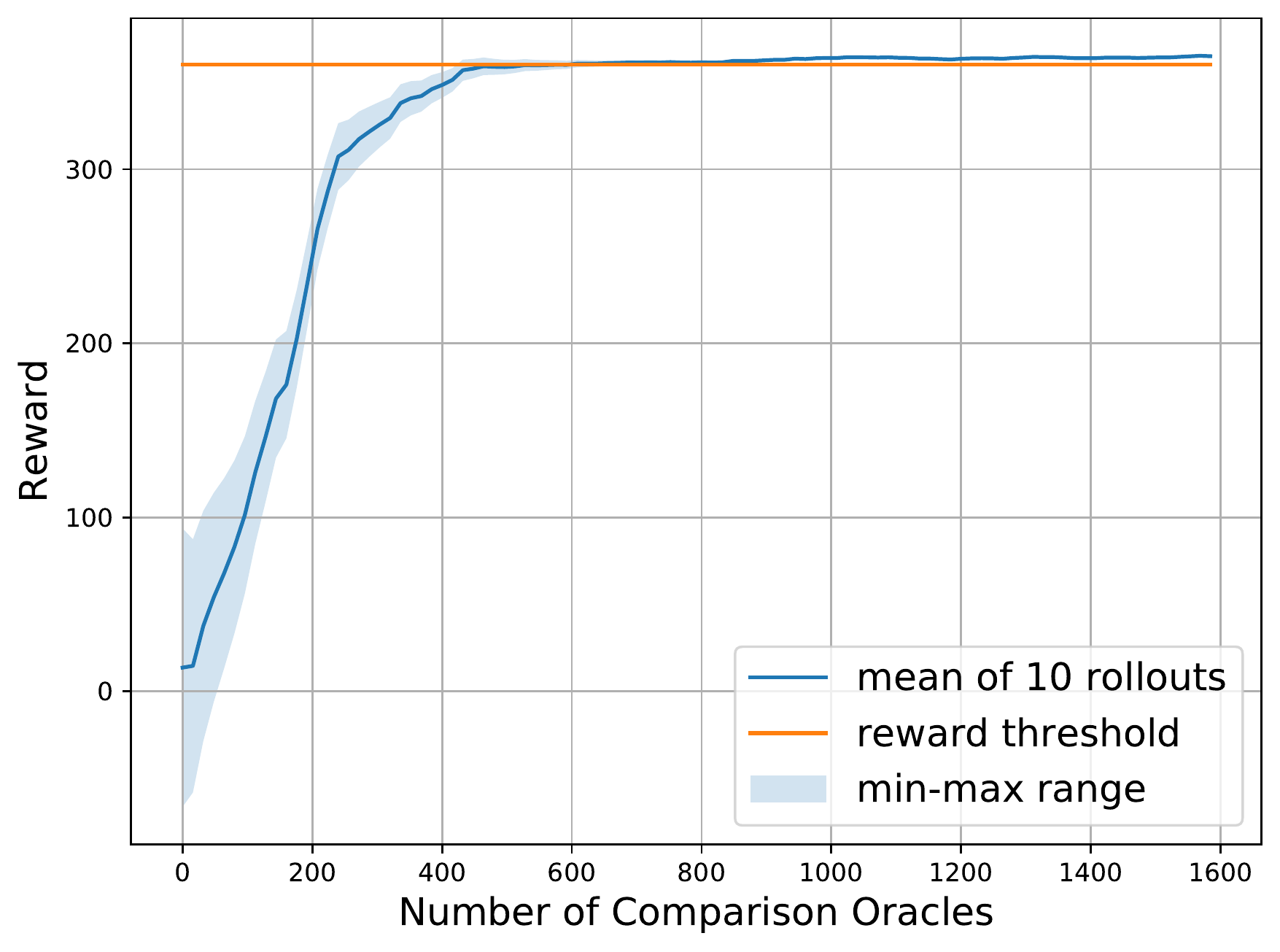}}\hfill
\subfloat{\includegraphics[width=.333\linewidth]{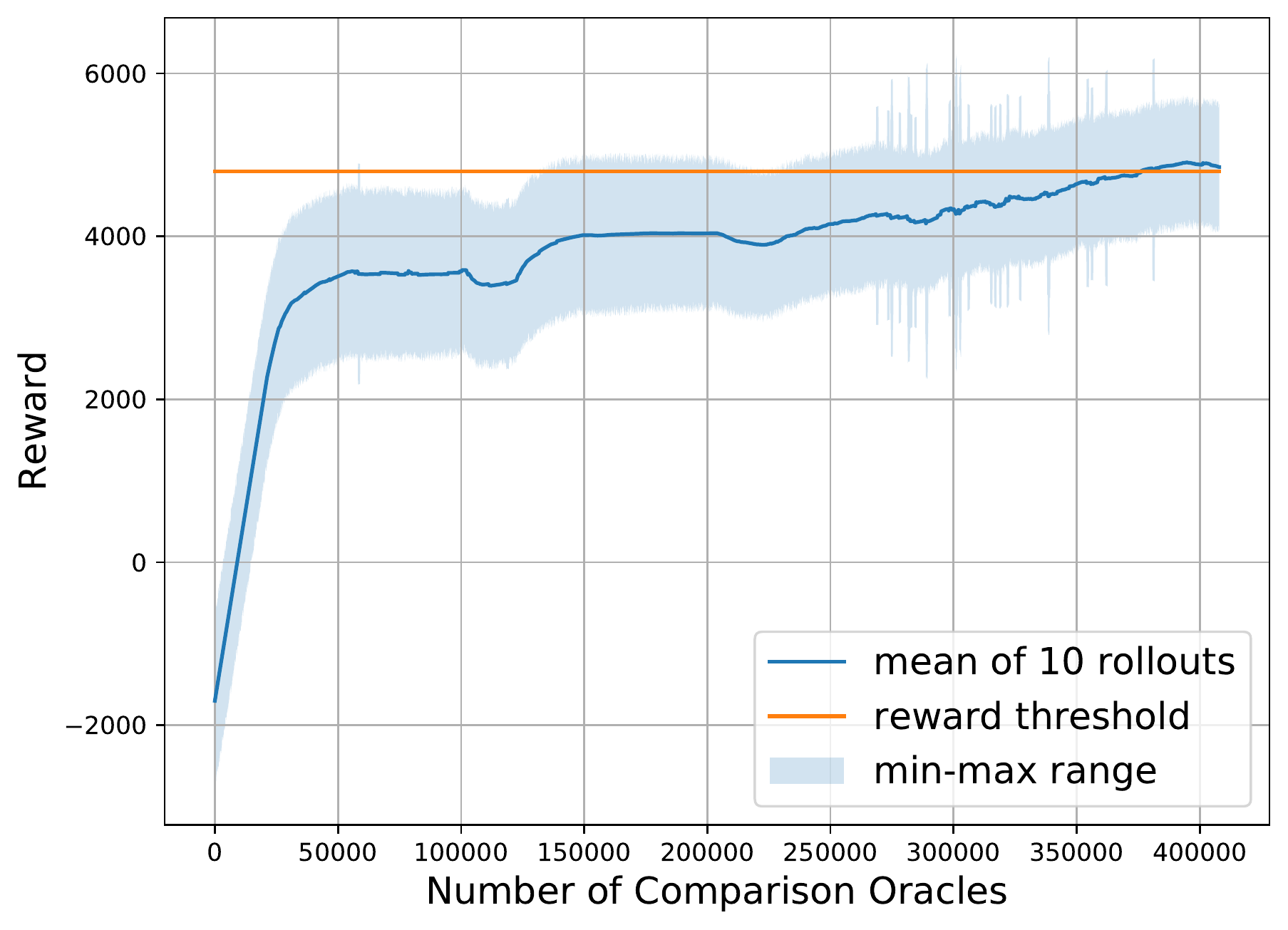}}\hfill
\subfloat{\includegraphics[width=.333\linewidth]{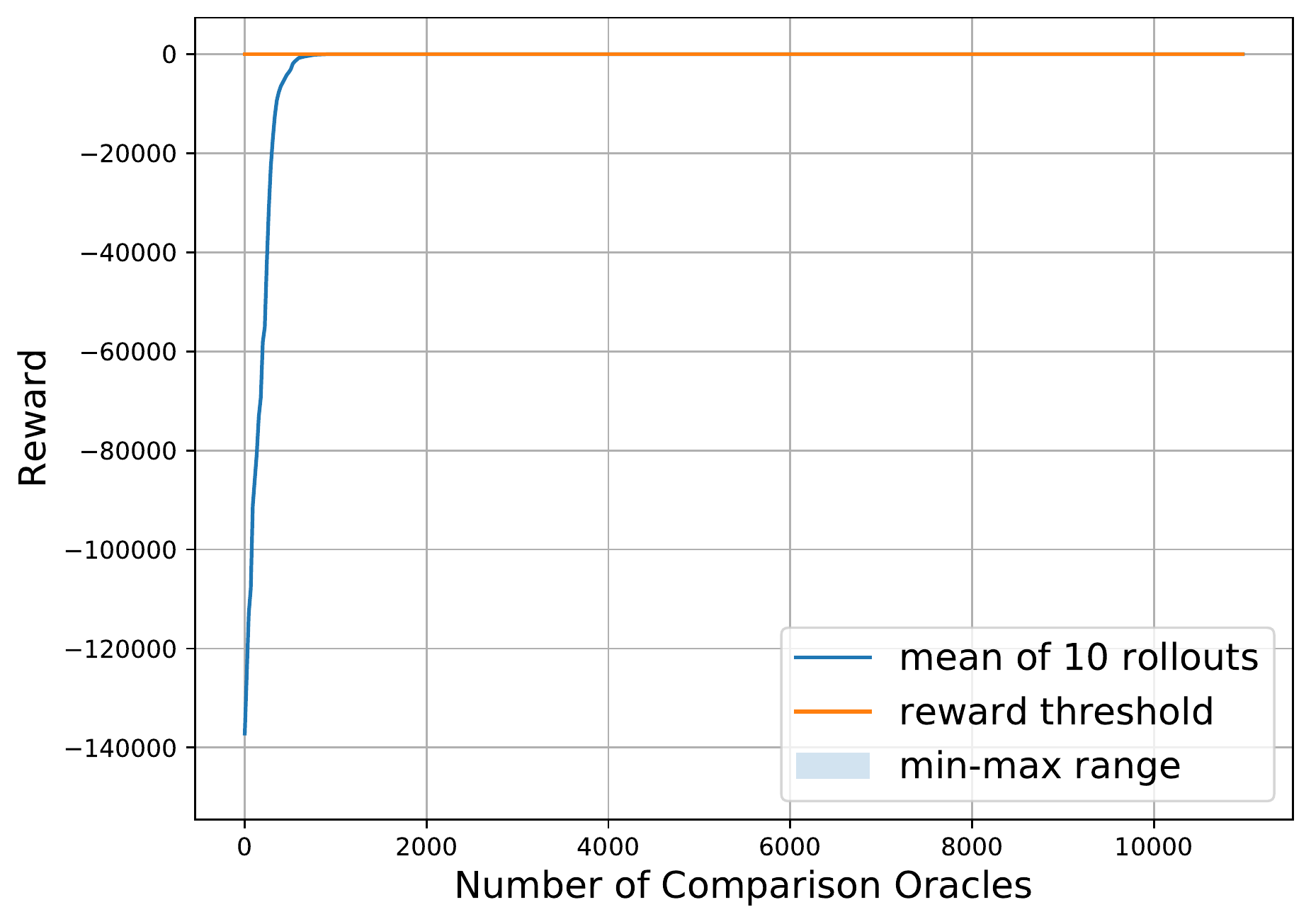}} 
\hfill
\vspace{-0.05in}
\caption{Rewards {\em v.s.} number of comparison oracles. Blue solid lines and shaded regions represent mean and +/- sigma of rewards.
{\bf Left:} Swimmer-v2. {\bf Middle:} HalfCheetah-v2. {\bf Right:} Reacher-v2. } \label{fig:mojoco}
\end{figure}

\section{Conclusion}
\label{sec:Conclusion}
In this paper, we have explored an intriguing connection between two seemingly unrelated areas: comparison-based optimization and one-bit compressed sensing. We have shown, theoretically and experimentally, that by importing tools from one-bit compressed sensing to optimization one can design a faster and more efficient algorithm capable of exploiting the compressible structure of the objective function's gradients. We think that this gradient-as-signal paradigm holds great promise, and that future work might consider applying further techniques from signal processing to zeroth-order optimization. For example, if $\nabla^{2}f(\bfx)$ is low-rank (as is frequently assumed) could one use tools from matrix recovery to approximate it and incorporate this into an optimization scheme?

\bibliographystyle{unsrt}
\bibliography{ZerothOrderBib}

\appendix

\section{Additional Proofs}
\label{sec:AdditionalProofs}
Here, we provide proofs and supporting lemmas for the results of Sections~\ref{section:INGD} and \ref{sec:SCOBO}.

\subsection{Proofs for Section~\ref{section:INGD}}
\label{app:Proofs_INGD}
Throughout this section, we assume $\bfx_{k+1} = \bfx_{k} - \alpha \hat{\bfg}_k$ where $\hat{\bfg}_k \approx \frac{\bfg_k}{\|\bfg_k\|_{2}}$. Before proceeding, it is convenient to introduce the following notation:
\begin{align*}
 \mathbf{e}_k & := \frac{\bfg_k}{\|\bfg_k\|_{2}} - \hat{\bfg}_k, \\
 \Delta_k &= \|\bfx_k - \proj_{\star}(\bfx_k)\|_{2}.
\end{align*}
We shall use the following inequality repeatedly, so we isolate it as a lemma:

\begin{lemma} \label{lemma:Consequence_of_RSC}
Suppose that $f(\bfx) \in \mathcal{F}_{L,\nu,d}$ and $\|\bfe_k\|_{2} \leq \eta < \nu/L$. Then \RV{either $\Delta_{k+1} < \alpha\eta$ or
\begin{equation*}
\left(\Delta_{k+1} - \alpha\eta\right)^{2} \leq \Delta_k^{2} - \frac{\alpha\nu}{L}\Delta_k + \alpha^{2}.
\end{equation*}
}
\end{lemma}
\begin{proof}
\RV{If $\Delta_{k+1} \le \alpha\eta$ we are done, so suppose $\Delta_{k+1} \geq \alpha\eta$.} Observe that:
\begin{align*}
    \Delta_{k+1} & = \|\bfx_{k+1} - \proj_{\star}(\bfx_{k+1})\|_{2} \\
    & \leq \|\bfx_{k+1} - \proj_{\star}(\bfx_{k})\|_{2} \\
    & = \|\bfx_{k} - \alpha \hat{\bfg}_k - \proj_{\star}(\bfx_{k})\|_{2} \\
    &  = \|\bfx_k - \alpha\left(\frac{\bfg_k}{\|\bfg_k\|_{2}} \RV{-} \mathbf{e}_k\right) - \proj_{\star}(\bfx_k)\|_{2} \\
    & \leq \|\bfx_k - \alpha\frac{\bfg_k}{\|\bfg_k\|_{2}}- \proj_{\star}(\bfx_k)\|_{2} + \alpha\|\mathbf{e}_k\|_{2}\\
    & \leq \|\bfx_k - \alpha\frac{\bfg_k}{\|\bfg_k\|_{2}}- \proj_{\star}(\bfx_k)\|_{2} + \alpha\eta.
\end{align*}
\RV{Because $\Delta_{k+1} \geq \alpha\eta$ we have
\begin{equation}
    0 \leq \Delta_{k+1} - \alpha\eta \leq \|\bfx_k - \alpha\frac{\bfg_k}{\|\bfg_k\|_{2}}- \proj_{\star}(\bfx_k)\|_{2}
    \label{eq:Equation_19}
\end{equation}
and squaring both sides we obtain
\begin{equation}
\left(\Delta_{k+1} - \alpha\eta\right)^2 \leq \|\bfx_k - \alpha\frac{\bfg_k}{\|\bfg_k\|_{2}}- \proj_{\star}(\bfx_k)\|_{2}^2.
\label{eq:SquareBothSides}
\end{equation}}
We now handle the term on the right-hand side:
\begin{align}
 \|\bfx_k - \alpha\frac{\bfg_k}{\|\bfg_k\|_{2}}- \proj_{\star}(\bfx_k)\|_{2}^{2}   &= \|\bfx_k - \proj_{\star}(\bfx_k)\|_{2}^{2} -2\left\langle \alpha\frac{\bfg_k}{\|\bfg_k\|_{2}},\bfx_k - \proj_{\star}(\bfx_k)\right\rangle  + \alpha^{2}\left\|\frac{\bfg_k}{\|\bfg_k\|_{2}}\right\|_{2}^{2} \cr
    & \stackrel{(a)}{\leq}  \|\bfx_k - \proj_{\star}(\bfx_k)\|_{2}^{2} - \RV{\frac{\alpha}{\|\bfg_k\|_{2}}\nu}\|\bfx_k - \proj_{\star}(\bfx_k)\|_{2}^{2} + \alpha^{2} \cr
    & \stackrel{(b)}{\leq} \|\bfx_k - \proj_{\star}(\bfx_k)\|_{2}^{2} -\frac{\alpha\nu}{L}\|\bfx_k - \proj_{\star}(\bfx_k)\|_{2} + \alpha^2 \cr
    & = \Delta_{k}^{2} - \frac{\alpha\nu}{L}\Delta_k + \alpha^{2} , \label{eq:Bound_with_Delta_k}
\end{align}
where in (a) we have used the fact that restricted $\nu$-strong convexity implies \eqref{eq:Restricted_Secant} while in (b) we have used $L$-Lipschitz differentiability: $\|\bfg_k\|_{2} \leq L\|\bfx_k - \proj_{\star}(\bfx_k)\|_{2} \Rightarrow \frac{\|\bfx_k - \proj_{\star}(\bfx_k)\|_{2}}{\|\bfg_k\|_{2}} \geq \frac{1}{L}$. Combining equations \eqref{eq:SquareBothSides} and \eqref{eq:Bound_with_Delta_k} completes the proof.
\end{proof}

It is interesting to determine when Lemma~\ref{lemma:Consequence_of_RSC} guarantees descent {\em i.e.} $\Delta_{k+1} \leq \Delta_{k}$.

\begin{lemma}
\label{lemma:GuaranteeDescent}
Suppose that $f(\bfx) \in \mathcal{F}_{L,\nu,d}$ and $\|\bfe_k\|_{2} \leq \eta < \nu/(2L)$. Then $\Delta_{k+1}\leq \Delta_{k}$ as long as $\Delta_{k} \geq \alpha\rho^{\star}/2$ where 
\begin{equation}
    \RV{\rho^{\star} = \frac{1-\eta^{2}}{\frac{\nu}{2L} - \eta}}.
\end{equation}
\end{lemma}

\begin{proof}
\RV{We prove this lemma by considering two cases: $\Delta_{k+1} < \alpha\eta$ and $\Delta_{k+1} \geq \alpha\eta$. We consider the latter case first.
Assume $\Delta_{k+1} \geq \alpha\eta$ and suppose in addition we have:}
\begin{equation} \label{eq:Ineq_for_completing_square}
    - \frac{\alpha\nu}{L}\Delta_k + \alpha^2 \leq -2\alpha\eta\Delta_k + \alpha^{2}\eta^{2}.
\end{equation}
From Lemma~\ref{lemma:Consequence_of_RSC} one obtains:
\begin{align*}
 \left(\Delta_{k+1} - \alpha\eta \right)^{2} & \leq \Delta_k^{2} - \frac{\alpha\nu}{L}\Delta_k + \alpha^2 \\
 &\leq \Delta_k^{2} -2\alpha\eta\Delta_k + \alpha^{2}\eta^{2} = \left(\Delta_k - \alpha\eta\right)^{2},
\end{align*}
where we are using the fact that $\eta < \nu/(2L) \leq 1$. Hence $\Delta_{k+1}\leq \Delta_{k}$. Solving \eqref{eq:Ineq_for_completing_square} for $\Delta_k$, and assuming $\eta < \nu/(2L)$, we get the condition:
\begin{equation}
    \Delta_k \geq \frac{\alpha(1-\eta^{2})}{\frac{\nu}{L} - 2\eta} = \frac{\alpha\rho^{\star}}{2}.
    \label{eq:LowerBoundDelta}
\end{equation}
That is, \eqref{eq:LowerBoundDelta} implies \eqref{eq:Ineq_for_completing_square} which in turn implies $\Delta_{k+1} \leq \Delta_k$. Now suppose $\Delta_{k+1} < \alpha\eta$. Observe that:
\begin{equation}
    \frac{\alpha\rho^{\star}}{2} = \frac{\alpha(1-\eta^{2})}{\frac{\nu}{L} - 2\eta} \stackrel{(a)}{\geq} \frac{\alpha(1-\eta^{2})}{1 - \eta} = \frac{\alpha(1-\eta)(1+\eta)}{1 - \eta} = \alpha(1+\eta) \geq \alpha\eta ,
    \label{eq:rho_greater_than_eta}
\end{equation}
where (a) holds as $\nu/L \leq 1$. Hence if $\Delta_k \geq \alpha\rho^{\star}/2$ and $\Delta_{k+1} < \alpha\eta$ we again have $\Delta_{k+1}\leq \Delta_k$.
\end{proof}
The next lemma shows that once the iterates $\bfx_k$ are sufficently close to the set of minimizers ({\em i.e.} $\mathcal{X}$) they remain within a small neighborhood of $\bfx_k$. That is, the iterates do not ``escape''.
\begin{lemma}[No escape, after \citep{levy2016power}]
\label{lemma:No_Escape_Precise}
Suppose that $f(\bfx) \in \mathcal{F}_{L,\nu,d}$. Fix $K > 0$ and assume that $\|\bfe_k\|_{2} \leq \eta < \nu/(2L)$ for all $0\leq k \leq K-1$ satisfying $\|\Delta_k\|_2 \geq \alpha\rho^{\star}$. If, for any $k <K$, we have that $\Delta_k \leq \alpha(1+\rho^{\star})$ then:
$$
\Delta_{k+t} \leq \alpha(1+\rho^{\star}) \quad \textnormal{ for all } 0 \leq t \leq K-k,
$$
\RV{where $\rho^{\star}$ is as in Lemma~\ref{lemma:GuaranteeDescent}.}
\end{lemma}

\begin{proof}
\RV{By Lemma~\ref{lemma:GuaranteeDescent}}, we have that if $\Delta_k \geq \alpha\rho^{\star} \geq \alpha\rho^{\star}/2$ then $\Delta_{k+1} \leq \Delta_k$. On the other hand, if $\Delta_{k} < \alpha\rho^{\star}$ then:
\begin{align*}
\Delta_{k+1} & = \|\bfx_{k+1} - \proj_{\star}(\bfx_{k+1})\|_{2} \leq \|\bfx_{k+1} - \proj_{\star}(\bfx_{k})\|_{2}  \\
        & = \|\bfx_{k} - \alpha\hat{\bfg}_k - \proj_{\star}(\bfx_{k})\|_{2} \\
    & \leq \|\bfx_{k} - \proj_{\star}(\bfx_{k})\|_{2} + \alpha \|\hat{\bfg}_k\|_{2} \\
    &\stackrel{(a)}{\leq} \Delta_{k} + \alpha \leq \alpha(1 + \rho^{\star}),
\end{align*}
\RV{where in (a) we are using the fact $\|\hat{\bfg}_k\|_{2} \leq 1$}. Thus, we obtain:
\begin{equation*}
\Delta_{k+1} \leq \max\{\Delta_k,\alpha(1+\rho^{\star})\} \quad\textnormal{ if } \left\|\bfe_k\right\|_{2} \leq \eta .
\end{equation*}
From this, it is easy to deduce that if $\Delta_k \leq \alpha(1+\rho^{\star})$ then $\Delta_{k+1}\leq \alpha(1+\rho^{\star})$, and the lemma follows by induction.
\end{proof}
We now prove an elementary lemma that quantifies the {\em rate of descent} of sequences satisfying the type of recurrence as in Lemma~\ref{lemma:Consequence_of_RSC}.

 \begin{lemma}[Sequence analysis]\label{prop:seq}
Consider a sequence $e_k\ge 0$ obeying $e^{2}_{k+1}\le e_k^{2} - a e_k+b$ for $k = 0,1,\ldots$ where $a,b >0$. We have
\begin{align*}
e_k \le \frac{\sqrt{2}e_{0}^{3/2}}{\sqrt{2e_0 + ak}},\quad k\in \{t:e_{0},\dots,e_{t+1} \ge 2b/a\}.
\end{align*}
\end{lemma}
\begin{proof}
Suppose that $e_{k} \geq 2b/a$, then $e_{k+1}^{2} \leq e_{k}^{2} - b \leq e_{k}^{2}$. Dividing both sides of  $e^{2}_{k+1}\le e_k^{2} - a e_k+b$ by $e_{k+1}^{2}e_{k}^{2}$ we obtain:
\begin{align*}
    \frac{1}{e_{k}^{2}} & \leq \frac{1}{e_{k+1}^{2}} - \frac{a}{e_{k+1}^{2}e_{k}} + \frac{b}{e_{k+1}^{2}e_{k}^{2}} 
    \leq \frac{1}{e_{k+1}^{2}} - \frac{a}{e_{k+1}^{2}e_{k}} + \frac{a}{2e_{k+1}^{2}e_{k}} \\
\Rightarrow \frac{1}{e_{k}^{2}} & \leq \frac{1}{e_{k+1}^{2}} - \frac{a}{2e_{k}^{3}} \stackrel{(a)}{\leq} \frac{1}{e_{k+1}^{2}} - \frac{a}{2e_{0}^{3}} \\
\Rightarrow  0 & \leq \frac{1}{e_{0}^{2}} \leq \frac{1}{e_{K}^{2}} - \frac{aK}{2 e_{0}^{3}} \quad \textnormal{ (by summing)} \\
\Rightarrow e_{K}^{2} & \leq \frac{2e_{0}^{3}}{2e_0 + aK} \Rightarrow  e_K \leq \frac{\sqrt{2}e_{0}^{3/2}}{\sqrt{2e_0 + aK}} ,
\end{align*}
where (a) follows from the fact that the sequence is decreasing.
\end{proof}

We now apply these results to deduce Theorem~\ref{thm:INGD_Convergence}:

\begin{proof}[Proof of Theorem~\ref{thm:INGD_Convergence}]
It suffices to show $\Delta_{K} \leq \alpha(1+\rho^{\star})$ as then by the $L$-Lipschitz differentiability of $f$:
$$
f(\bfx_{K}) - f^\star \leq \frac{L}{2}\Delta_K^{2} = \frac{L\alpha^{2}(1+\rho^{\star})^{2}}{2} .
$$
By Lemma~\ref{lemma:No_Escape_Precise}, it actually suffices to show $\Delta_{\ell} \leq \alpha(1+\rho^{\star})$ for any $0\leq \ell \leq K$. So, fix any $k$ and observe that if $\Delta_{k+1} \leq \alpha\eta$ we are done as $\alpha\eta \leq \alpha\rho^{\star} \leq \alpha(1+\rho^{\star})$ by \eqref{eq:rho_greater_than_eta}. Thus, assume $\Delta_{k+1}\geq \alpha\eta$  for all $0\leq k \leq K-1$ whence by Lemma~\ref{lemma:Consequence_of_RSC}:
\begin{equation} \label{eq:Delta_k_recurrence2}
\left(\Delta_{k+1} - \alpha\eta \right)^{2} \leq \Delta_k^{2} - \frac{\alpha\nu}{L}\Delta_k + \alpha^2.
\end{equation}
Let $e_{k} = \Delta_{k} -\alpha\eta \geq 0$. One may rewrite \eqref{eq:Delta_k_recurrence2} as:
\begin{equation*}
e_{k+1}^{2} \leq e_{k}^{2} - \underbrace{\left(\frac{\alpha\nu}{L} - 2\alpha\eta \right)}_{= a}e_{k} + \underbrace{\alpha^{2}\left(\eta^{2} - \frac{\eta\nu}{L} + 1 \right)}_{=b}.
\end{equation*}
Observe that if $e_{k} <2b/a$ for any $0\leq k \leq K-1$ then by substituting in the definitions of $e_k,a$ and $b$:
\begin{align*}
   & \Delta_k - \alpha\eta < \frac{2\alpha^{2}(\eta^{2} - \frac{\eta\nu}{L} + 1)}{\frac{\alpha\nu}{L} - 2\alpha\eta} = \alpha\frac{2\eta^{2} - \frac{2\eta\nu}{L} + 2}{\frac{\nu}{L} - 2\eta} \\
 \Rightarrow & \Delta_k < \alpha\left(\frac{2\eta^{2} - \frac{2\eta\nu}{L} + 2}{\frac{\nu}{L} - 2\eta} + \eta \right) = \alpha\left(\frac{2 - \frac{\eta\nu}{L}}{\frac{\nu}{L} - 2\eta}\right) \stackrel{(a)}{\leq} \alpha\left(\frac{2 -2\eta^2}{\frac{\nu}{L} - 2\eta}\right) = \alpha\rho^{\star} \leq \alpha(1+\rho^{\star}),
\end{align*}
where (a) follows as $-\eta < -\nu/(2L)$ and again we are done. So, assume $e_{k} \geq 2b/a$ for $0\leq k \leq K-1$. From Lemma~\ref{prop:seq} we then obtain:
\begin{align*}
    \Delta_{K} - \alpha\eta & \leq \frac{\sqrt{2}\left(\Delta_0 - \alpha\eta\right)^{3/2}}{\sqrt{2\left(\Delta_0 - \alpha\eta\right) + \left(\frac{\alpha\nu}{L} - 2\alpha\eta \right)K}} \leq \frac{\sqrt{2}\left(\Delta_0 - \alpha\eta\right)^{3/2}}{\sqrt{ \left(\frac{\alpha\nu}{L} - 2\alpha\eta \right)K}} \stackrel{(a)}{=} \alpha\rho^{\star} ,
\end{align*}
where (a) follows from the choice of $K$. This implies $\Delta_K \leq \alpha(\eta +\rho^{\star}) \leq \alpha(1+\rho^{\star})$ and we are done. 
\end{proof}

\subsection{Proofs for Section~\ref{sec:SCOBO}}
\label{sec:SCOBO_Proofs}

\begin{proof}[Proof of Theorem~\ref{thm:FormalMainTheorem_2}]
If $\Delta_{k} \geq \alpha\rho^{\star}$ then as $f$ is restricted $\nu$-strongly convex: $\|\bfg_k\|_{2} \geq \alpha\nu\rho^{\star}/2$, see \eqref{eq:GEB}. Now appealing to Theorem~\ref{theorem:Grad_Estimate_Error_Bound} for the prescribed choice of $m$, we obtain $\left\|\hat{\bfg}_k - \bfg_k/\|\bfg_k\|_{2}\,\right\|_{2} \leq \eta$ with probability at least $1 - 8\exp\left(-c\eta^{4}s\log(2d/s)\right)$. Use the union bound to conclude that 
$$
\left\|\hat{\bfg}_k - \bfg_k/\|\bfg_k\|_{2}\,\right\|_{2} \leq \eta \quad\textnormal{ for } 0 \leq k \leq K-1 \text{ with } \Delta_{k} \geq \alpha\rho^{\star}
$$
with probability greater than $1 - 8K\exp\left(-c\eta^{4}s\log(2d/s)\right)$. Conditional on this, we apply Theorem~\ref{thm:INGD_Convergence} to obtain:
$$
f(\bfx_{K}) - f^\star\leq \frac{L}{2}\alpha^{2}(1+\rho^{\star})^{2}
$$
as desired.
\end{proof}

Choosing $\alpha$ such that
\begin{equation}
\varepsilon = \frac{L}{2}\alpha^{2}(1+\rho^{\star})^{2}
\label{eq:Relating_alpha_and_epsilon}
\end{equation}
yields the results as informally stated in Section~\ref{sec:SCOBO}.

\begin{proof}[Proof of Corollary~\ref{lemma:SCOBO_Guaranteed_Descent}]
As $f$ is $L$-Lipschitz differentiable and $\bfx_{k+1} = \bfx_{k} - \alpha\hat{\bfg}_k$:
\begin{align*}
f(\bfx_{k+1}) & \leq f(\bfx_{k}) - \alpha\|\bfg_k\|_{2}\langle \hat{\bfg}_k,\bfg_k/\|\bfg_k\|_{2}\rangle  + \frac{L}{2}\alpha^{2} \\
    & \stackrel{(a)}{\leq} f(\bfx_k) - \alpha\|\bfg_k\|_{2}\RV{\left(1 - \eta\right)} + \frac{L\alpha^2}{2} ,
\end{align*}
where (a) follows from $\|\hat{\bfg}_k - \bfg_k/\|\bfg_k\|_{2}\,\|_{2} \leq \eta$. Hence, $f(\bfx_{k+1}) \leq f(\bfx_k)$ as long as:
\begin{equation} \label{eq:Goth_Star}
\|\bfg_k\|_{2} \geq \frac{L\alpha}{2(1-\eta)}.
\end{equation}
So, suppose that \eqref{eq:Goth_Star} does not hold. As $f$ is restricted $\nu$-strongly $\|\bfg_k\|_2 \geq \frac{\nu}{2}\Delta_k$, see \eqref{eq:GEB}. Combining this with $L$-Lipschitz differentiability:
\begin{align*}
f(\bfx_k) - f^\star & \leq \frac{L}{2}\Delta_k^{2} \leq \frac{L}{2}\frac{4}{\nu^2}\|\bfg_k\|_{2}^{2} \\
    & \leq \frac{L^{3}\alpha^{2}}{2\nu^2(1-\eta)^2} \\
    & = \left[\frac{L}{2}\alpha\right] \left[ \frac{1}{(\nu/L)^2(1-\eta)^2}\right] \stackrel{(a)}{\leq} \left[\frac{L}{2}\alpha^2\right] \left[(1+\rho^{\star})^2\right] \stackrel{(b)}{=} \varepsilon ,
\end{align*}
where (a) is \RV{shown separately as Lemma~\ref{lm:ineq(a)}} 
and (b) holds as $\alpha$ is chosen as in \eqref{eq:Relating_alpha_and_epsilon}. 
\end{proof}

\begin{lemma} \label{lm:ineq(a)}
\RV{Given $0<\eta<\frac{\nu}{2L}\leq 0.5$, it holds
\begin{equation*}
    \frac{1}{(\nu/L)^2(1-\eta)^2} < (1+\rho^\star)^2.
\end{equation*}
}
\end{lemma}
\begin{proof}
\RV{
Denote $p:=\frac{\nu}{2L}$. Recall that
\begin{align*}
        \rho^\star = \frac{1-\eta^{2}}{\frac{\nu}{2L} - \eta} = \frac{1-\eta^2}{p-\eta}.
\end{align*}
Since $\frac{1}{(\nu/L)(1-\eta)}>0$ and $1+\rho^\star>0$, so it is equivalent to show 
\begin{equation*}
    \frac{1}{(\nu/L)(1-\eta)} < 1+\rho^\star.
\end{equation*}
Notice that
\begin{align*}
    1+\rho^\star-\frac{1}{(\nu/L)(1-\eta)} 
    &= \frac{p-\eta+1-\eta^2}{p-\eta}-\frac{1}{2p-2p\eta} \\
    &= \frac{2p^2+p+\eta+2p\eta^3-4p\eta-2p^2\eta}{(p-\eta)(2p-2p\eta)}. 
\end{align*}
By $0<\eta<p:=\frac{\nu}{2L}\leq 0.5$, we have $2p^2>2p^2\eta$ and $p+\eta>4p\eta$. Thus, the numerator is positive, {\em i.e.}
\begin{align*}
    2p^2+p+\eta+2p\eta^3-4p\eta-2p^2\eta>0.
\end{align*}
Together with the denominator $(p-\eta)(2p-2p\eta)>0$, we conclude
\begin{align*}
    1+\rho^\star-\frac{1}{(\nu/L)(1-\eta)}  > 0 .
\end{align*}
This finishes the proof.
}
\end{proof}

\section{Working with Gaussian measurement vectors}

In stating Theorem~\ref{theorem:PlanVershynin}, we claim it is possible to use $\bfz_i$ sampled uniformly from $\mathbb{S}^{d-1}$, instead of using Gaussian $\bfz_i$ as studied in \cite{Plan2012}. As this extension is somewhat ``folklore'', here we describe how one can also use Gaussian $\bfz_i$ within SCOBO. Specifically, one can do the following:
\begin{enumerate}
    \item Sample $\bfz_1,\ldots,\bfz_m$ from the Gaussian distribution with mean $\mathbf{0}$ and covariance matrix $I$. \\
    
    \item For $i=1,\ldots, m$ define $\hat{\bfz}_i = \bfz_i/\|\bfz_i\|_2$. Then, $\hat{\bfz}_i$ are sampled uniformly from $\mathbb{S}^{d-1}$. \\
    
    \item Sample from the comparison oracle using $\hat{\bfz}_i$, {\em i.e.} $y_i = \mathcal{C}_f(x, x+r\hat{\bfz}_i)$. \\
    
    \item The arguments of Section~\ref{sec:GradEst} apply unchanged to yield $y_i = \xi_i\sign\left(\hat{\bfz}_i^\top\bfg\right)$ with $\mathbb{P}[\xi_i=1] :=p \geq 0.5 + \delta_0/2$ and $\mathbb{P}[\xi_i=-1] := 1-p \leq 0.5 - \delta_0/2$. \\
    
    \item Crucially, observe {\em scaling does not change the sign}. So $\sign\left(\hat{\bfz}_i^\top\bfg\right) = \sign\left(\bfz_i^\top\bfg\right)$ and thus $y_i = \xi_i\sign\left(\bfz_i^\top\bfg\right)$. \\
    
   \item Now solve the problem 
        \begin{equation}
            \hat{\bfx} := \argmax_{\|\bfx'\|_{1} \leq \sqrt{s} \textnormal{ and }\|\bfx'\|_{2} \leq 1} \sum\nolimits_{i=1}^{m} y_{i}\bfz_{i}^{\top}\bfx' 
        \end{equation}
    and note that the $\mathbf{z}_i$ are now Gaussian random vectors. The results of \cite{Plan2012} now apply without modification to yield the claim of Theorem~\ref{theorem:Grad_Estimate_Error_Bound}
    \end{enumerate}
    To summarize, one may ``cheat'' slightly by using $\hat{\mathbf{z}}_i$ when querying the oracle but the Gaussian $\mathbf{z}_i$ when solving the recovery problem. Because $\text{sign}(\hat{\mathbf{z}}_i^{\top}\mathbf{g})= \text{sign}(\mathbf{z}_i^{\top}\mathbf{g})$ it doesn't matter. \\
\end{document}